\documentclass[11pt]{amsart}
\usepackage{amsmath,amssymb,amsthm}
\usepackage{color,graphics,srcltx,float,anysize}
\usepackage{stmaryrd}
\usepackage{enumerate}
\usepackage{booktabs}
\usepackage[table,x11names]{xcolor}
\usepackage{longtable}
\usepackage{makecell}
\usepackage{hyperref}
\usepackage{makecell}

\definecolor{darkblue}{rgb}{0,0,0.8}

\newtheorem{theorem}{Theorem}[section]
\newtheorem{lemma}[theorem]{Lemma}
\newtheorem{corollary}[theorem]{Corollary}

\theoremstyle{definition}
\newtheorem{defn}[theorem]{Definition}
\newtheorem{example}[theorem]{Example}
\newtheorem{question}[theorem]{Question}
\newtheorem{problem}[theorem]{Problem}
\newtheorem{hypothesis}[theorem]{Hypothesis}

\DeclareMathOperator{\PSL}{{\mathrm{PSL}}}
\DeclareMathOperator{\PSU}{{\mathrm{PSU}}}
\DeclareMathOperator{\PGU}{{\mathrm{PGU}}}
\DeclareMathOperator{\PSp}{{\mathrm{PSp}}}

\DeclareMathOperator{\PGL}{{\mathrm{PGL}}}
\DeclareMathOperator{\AGL}{{\mathrm{AGL}}}

\DeclareMathOperator{\SL}{{\mathrm{SL}}}
\DeclareMathOperator{\GL}{{\mathrm{GL}}}
\DeclareMathOperator{\GU}{{\mathrm{GU}}}
\DeclareMathOperator{\SU}{{\mathrm{SU}}}
\DeclareMathOperator{\Sp}{{\mathrm{Sp}}}

\DeclareMathOperator{\GammaL}{\Gamma{\mathrm{L}}}
\DeclareMathOperator{\PGammaL}{\mathrm{P}\Gamma{\mathrm{L}}}

\DeclareMathOperator{\PGammaU}{\mathrm{P}\Gamma{\mathrm{U}}}

\DeclareMathOperator{\GO}{{\mathrm{GO}}}
\DeclareMathOperator{\Ree}{{\mathrm{Ree}}}
\DeclareMathOperator{\Sz}{{\mathrm{Sz}}}

\DeclareMathOperator{\GF}{{\mathrm{GF}}}
\DeclareMathOperator{\Aut}{{\mathrm{Aut}}}
\DeclareMathOperator{\Sym}{{\mathrm{Sym}}}
\DeclareMathOperator{\Syl}{{\mathrm{Syl}}}
\DeclareMathOperator{\soc}{{\mathrm{soc}}}
\DeclareMathOperator{\fix}{{\mathrm{fix}}}
\DeclareMathOperator{\Inn}{{\mathrm{Inn}}}
\DeclareMathOperator{\Out}{{\mathrm{Out}}}

\DeclareMathOperator{\Wr}{\,{\mathrm{wr}\,}}

\newcommand\norml{\trianglelefteqslant}

\newcommand{\MAGMA}{{\sc Magma}\ } 
\renewcommand{\leq}{\leqslant}
\renewcommand{\geq}{\geqslant}

\renewcommand{\ge}{\geqslant}

\usepackage{color}

\allowdisplaybreaks

\begin{document}
\title{Orbits of Sylow subgroups of finite permutation groups}
\author[Bamberg]{John Bamberg}
\author[Bors]{Alexander Bors}
\author[Devillers]{Alice Devillers}
\author[Giudici]{Michael Giudici}
\author[Praeger]{Cheryl E. Praeger}
\author[Royle]{Gordon F. Royle}

\address{All but second author: Centre for the Mathematics of Symmetry and Computation, Department of Mathematics and Statistics, The University of Western Australia, Crawley, WA 6009, Australia.}
\address{Second author: School of Mathematics and Statistics, Carleton University, 1125 Colonel By Drive, Ottawa ON K1S 5B6, Canada.}
\email{All but second author: firstname.lastname@uwa.edu.au}
\email{Second author: alexanderbors@cunet.carleton.ca}
\date{\today}

\begin{abstract}
We say that a finite group $G$ acting on a set $\Omega$ has \emph{Property~$(*)_p$} for a prime $p$ if $P_\omega$ is a Sylow $p$-subgroup of $G_\omega$ for all $\omega\in\Omega$ and Sylow $p$-subgroups $P$ of $G$.
Property~$(*)_p$ arose in the recent work of Tornier (2018) on local Sylow $p$-subgroups of Burger-Mozes groups, and he determined the values of $p$ for which
the alternating group $A_n$ and symmetric group $S_n$ acting on $n$ points has Property~$(*)_p$. In this paper, we extend this result to finite $2$-transitive groups
and we give a structural characterisation result for the finite primitive groups that satisfy Property~$(*)_p$ for an allowable prime~$p$.
\end{abstract}

\maketitle

\section{Introduction}
Given $G\leqslant \Sym(\Omega)$ with $|\Omega|=n$, the \emph{Burger-Mozes group} $U(G)$ \cite{BM} is the group of all automorphisms of the $n$-regular tree $T_n$ such that, for all vertices $v$, the stabiliser of $v$ induces the group $G$ on the set of all neighbours of $v$. Tornier \cite{tornier} introduced the following property when studying local Sylow $p$-subgroups of Burger-Mozes groups. We use $\Syl_p(G)$ to denote the set of all Sylow $p$-subgroups of a group $G$ and for a group $G$ acting on a set $\Omega$ we use $G_\omega$ to denote the stabiliser in $G$ of the point $\omega\in\Omega$.

\begin{defn}\label{def}
Let $G$ be a finite group acting on a set $\Omega$ with $|\Omega|=n$, let $p$ be a prime and let $P\in\Syl_p(G)$. We say that $G$ has Property~$(*)_p$ if $P_\omega\in\Syl_p(G_\omega)$ for all $\omega\in\Omega$.
\end{defn}
 
Note that since $G$ acts transitively by conjugation on the set of Sylow $p$-subgroups of $G$, Property~$(*)_p$ does not depend on the choice of $P$. From the definition it is immediate that any \emph{semiregula}r permutation group, (i.e., one where $G_\omega=1$ for all $\omega\in\Omega$) has Property~$(*)_p$. 

Tornier's motivation for studying Property~$(*)_p$ was as follows:
Let $T$ be a finite subtree of $T_n$ and for any $H\leqslant\Aut(T_n)$ let $H_{(T)}$ denote the pointwise stabiliser of $T$ in $H$. Then Tornier \cite[Proposition 10]{tornier} showed that for $G\leqslant\Sym(\Omega)$ and $P\in\Syl_p(G)$, the group $U(P)_{(T)}$ is a local Sylow $p$-subgroup of $U(G)_{(T)}$ if and only if $G$ has Property~$(*)_p$.  Tornier \cite[Proposition 11]{tornier} showed that if $G$ is the alternating group $A_n$ or symmetric group $S_n$ acting on $n$ points then $G$ has Property~$(*)_p$, with $p$ dividing $|G|$, if and only if $n_pp>n$, where $n_p$ is the highest power of $p$ that divides $n$.  Moreover, if either $P$ has the same orbits as $G$ or $|\Omega|=p^n$, then $(*)_p$ holds \cite[Propositions 12 and 13]{tornier}.

The purpose of this paper is to undertake a thorough investigation of permutation groups with Property~$(*)_p$. In Section~\ref{sec:prelim} we start by showing that a permutation group has Property $(*)_p$ if and only if its induced action on each orbit has Property $(*)_p$, thereby reducing the problem to transitive groups. In the remainder of Section~\ref{sec:prelim} we derive some basic structural results leading up to the fact that a transitive permutation group of degree $n$ has Property $(*)_p$ if and only it has a Sylow $p$-subgroup all of whose orbits are the same size. In this case, the common size of the orbits is the highest power of $p$ dividing $n$,  that is,  $n_p$. This also allows us to deduce that if $pn_p>n$ then $G$ has Property $(*)_p$ (Lemma \ref{lem:elementary}(5)).

A transitive group can either be imprimitive or primitive, and so we consider each case in turn. In Section~\ref{section:imprimitive} we consider imprimitive groups, where we prove the following:

\begin{theorem}\label{thm:imprim}
Let $G$ act transitively on a set $\Omega$ with system of imprimitivity $\mathcal{B}$. Let $B\in\mathcal{B}$, let $G_B^B$ be the permutation group induced on $B$ by the setwise stabiliser $G_B$ and let $G^{\mathcal{B}}$ be the permutation group induced by $G$ on $\mathcal{B}$. 
\begin{enumerate}
\item If $G_B^B$ and $G^{\mathcal{B}}$ have Property $(*)_p$ then $G$ has Property $(*)_p$.
\item If $G$ has Property $(*)_p$ then $G_B^B$ has Property $(*)_p$.
\item If $G$ has Property $(*)_p$ then $G^{\mathcal{B}}$ does not necessarily have Property $(*)_p$.
\end{enumerate}
\end{theorem}

Then we turn to primitive groups, and in Section~\ref{sec:primitive} and Section~\ref{sec:primitive-proof} we develop, and then prove, the following structural characterisation of primitive groups with Property~$(*)_p$.
\begin{theorem}\label{thm:primitive}
Let $G$ be a primitive permutation group on $\Omega$ and suppose that $G$ has Property~$(*)_p$ for some prime $p$ dividing $|\Omega|$. Then one of the following holds:
\begin{enumerate}
 \item $G$ is an almost simple group;
\item $G$ is of Affine type and $|\Omega|=p^k$;
\item $\Omega=\Delta^k$ for some $k\geq 2$ and $G\leqslant H\Wr K$ where $H$ is an almost simple group acting primitively on $\Delta$ with Property~$(*)_p$ and $K\leqslant S_k$. Moreover, either $p$ is coprime to $|K|$, or $p$ divides $|K|$ and $|\Delta|$ is a power of $p$.
\end{enumerate}
Moreover, any primitive group in cases (2) and (3) has Property~$(*)_p$.
\end{theorem}

Not all almost simple groups satisfy Property~$(*)_p$ and indeed examples for which $pn_p<n$ and $p$ divides $|G_\omega|$ seem rare. Indeed, apart from examples from one infinite family satisfying Property $(*)_2$, there are only 6 such examples of degree at most 4095. See Subsection \ref{sec:smalldegree}. This suggests the following problem:

\begin{problem}\label{prob:prim}
Determine all the almost simple primitive permutation groups $G$ of degree $n$ that have Property $(*)_p$ for some prime $p$ dividing $n$ and for which $pn_p<n$ and $p$ divides $|G_\omega|$.
\end{problem}

As a contribution towards Problem \ref{prob:prim}, we determine all the 2-transitive permutation groups with Property~$(*)_p$. This generalises Tornier's result for $A_n$ and $S_n$, while all 2-transitive groups with Property~$(*)_p$ and $n_p=p$ have already been determined by the fourth author \cite[Theorem]{Praeger}. 

\begin{theorem}\label{thm:2trans}
Let $G$ be a $2$-transitive permutation group of degree $n$ on a set $\Omega$ with $\omega\in\Omega$, and let $p$ a prime dividing $n$. Then $G$ has Property~$(*)_p$ if and only if  one of the following holds:
\begin{enumerate}[{\rm (a)}]
    \item $pn_p>n$;
\item $p$ does not divide $|G_\omega|$;
    \item $G=A_5$ with $n=6$ and $p=2$;
    \item  $G=M_{11}$ with $n=12$ and $p=3$;
    \item $G=\PGammaL_2(8)$ with $n=28$ and $p=2$;
  \end{enumerate}
\end{theorem}
We note that in case (e) the subgroup $\PSL_2(8)$ of $G$ also has Property~$(*)_p$. However, $\PSL_2(8)$ is not 2-transitive of degree $28$, and so does not appear as a possibility.

We finish the paper in Section \ref{sec:small} with a survey of the permutation groups of degree at most 39 that have Property~$(*)_2$ or $(*)_3$.

\section{Preliminaries}\label{sec:prelim}

Our first observation (Lemma~\ref{lem:allsame}) shows that, in order to study groups with  Property~$(*)_p$, it is sufficient to restrict our attention to transitive group actions. The proof of Lemma~\ref{lem:allsame} follows immediately from Definition~\ref{def}.

\begin{lemma}\label{lem:allsame}
Let $G$ act on a set $\Omega$. Then $G$ has Property~$(*)_p$ if and only if $G$ acts on each of its orbits with Property~$(*)_p$.
\end{lemma}

The following is a generalisation of a result of Wielandt \cite[Theorem 3.4$'$]{wielandt} to arbitrary group actions.

\begin{lemma}\label{lem:shortest}
Let $G$ be a group acting transitively on a set $\Omega$  of size $n$ and let $P\in\Syl_p(G)$. Then the minimum length of a  $P$-orbit is  $n_p$.
\end{lemma}
\begin{proof}
Let $\omega\in\Omega$ and let $Q\in\Syl_p(G_\omega)$. Then by Sylow's Theorems, there exists a Sylow $p$-subgroup $P'$ of $G$ such that $Q\leqslant P'$. Then $Q=P'\cap G_\omega=P'_{\omega}$ and so $P'$ has an orbit of length $|P':Q|=|G|_p/|G_\omega|_p=n_p|G_\omega|_p/|G_\omega|_p=n_p$. Since all Sylow $p$-subgroups are conjugate in 
 $G$ it follows that $P$ has an orbit of length $n_p$.
Moreover,  if $P$ has a shorter orbit $\omega^P$ then $|P_\omega|=|P|/|\omega^P|> |P|/n_p=|Q|$, contradicting $|Q|$ being the order of a Sylow $p$-subgroup of a point stabiliser. 
\end{proof}

This leads to the following easy test for Property $(*)_p$.

\begin{lemma}\label{lem:easy}
Let $G$ act transitively on $\Omega$, let $p$ be a prime, and let $P\in\Syl_p(G)$. Then  Property~$(*)_p$ holds if and only if all orbits of $P$ on $\Omega$ have the same length, and that length is  $n_p$.
\end{lemma}
\begin{proof}
By Lemma \ref{lem:shortest}, $P$ has an orbit of length $n_p$. Also if $|\omega^P|=n_p$ then $|P_\omega|=|P|/n_p=|G|_p/n_p=|G_\omega|_p$. Hence all orbits of $P$ on $\Omega$ have the same length if and only if $P_\omega$ has order $|G_\omega|_p$ for all $\omega\in\Omega$, that is, if and only if $P_\omega\in\Syl_p(G_\omega)$ for all~$\omega\in\Omega$.
\end{proof}

It is easy to see from Lemma \ref{lem:easy} combined with Lemma \ref{lem:allsame} that, if $G$ acts on $\Omega$ with Property $(*)_p$, then the permutation group $G^{\Omega}\leqslant\Sym(\Omega)$ induced by $G$ also has Property $(*)_p$. We collect some other easy observations in the following lemma.

\begin{lemma}\label{lem:elementary}
Let $G$ act transitively on a set $\Omega$ of size $n$, let $\omega\in\Omega$ and let $p$ be a prime.
\begin{enumerate}
\item If $p$ does not divide $|G_\omega|$ then $G$ has Property~$(*)_p$.
\item If $G$ acts faithfully on $\Omega$ with Property~$(*)_p$ such that $p$ divides $|G|$, then $p$ divides $n$.
\item If $|\Omega|=p^k$ then $G$ has Property~$(*)_p$ but not Property $(*)_r$ for any prime $r\neq p$ which divides $|G^\Omega|$, where $G^\Omega$ is the permutation group on $\Omega$ induced by $G$.
\item If $pn_p>n$ then $G$ has Property~$(*)_p$.
\item If $H$ is a transitive subgroup of $G$ and $G$ has Property~$(*)_p$ then $H$ has Property~$(*)_p$.
\end{enumerate}
\end{lemma}

\begin{proof}\leavevmode
\begin{enumerate}
\item If $p$ does not divide $|G_\omega|$ for some $\omega$ then $p$ does not divide $|G_\omega|$ for all $\omega$. Since $P_\omega$ is a $p$-subgroup of $G_\omega$ it follows that $P_\omega=1$ for all $\omega\in\Omega$ and is a Sylow $p$-subgroup of $G_\omega$. 
\item Suppose that $G$ has Property~$(*)_p$ with $p$ coprime to $n$. Then, by Lemma~\ref{lem:easy}, all orbits of a Sylow $p$-subgroup of $G$ have length $n_p=1$. Since $G$ acts faithfully on $\Omega$ it follows that a Sylow $p$-subgroup of $G$ is trivial and so $p$ does not divide $|G|$.
\item By Lemma \ref{lem:shortest} a Sylow $p$-subgroup is transitive and so Property~$(*)_p$ holds while the rest of the assertion follows from part (2).
\item By Lemma \ref{lem:shortest}, $P$ has an orbit of length $n_p$ and all orbits of $P$ have length at least $n_p$. Since orbits of $P$ have length a power of $p$ and $pn_p>n$ it follows that all orbits of $P$ have length $n_p$ and so by Lemma \ref{lem:easy} we have that $G$ has Property~$(*)_p$.
\item Let $P\in\Syl_p(H)$. Then by Lemma \ref{lem:shortest}, all orbits of $P$ have size at least $n_p$. However, since $G$ has Property~$(*)_p$ and $P$ is contained in a Sylow $p$-subgroup of $G$, Lemma \ref{lem:easy} implies that every orbit of $P$ has size at most $n_p$. Thus all orbits of $P$ have length $n_p$ and the result follows from applying Lemma \ref{lem:easy} again.
\qedhere
\end{enumerate}
\end{proof}

\section{Imprimitive groups}\label{section:imprimitive}

In this section, we investigate imprimitive permutation groups with Property $(*)_p$. Suppose that  $G$ is an imprimitive group with a non-trivial system of imprimitivity $\mathcal{B}$ and that $B \in \mathcal{B}$. As in Theorem~\ref{thm:imprim}, we let $G_B^B$ denote the permutation group induced on a block $B$ by the setwise stabiliser $G_B$, and $G^\mathcal{B}$ denote the permutation group induced by $G$ on $\mathcal{B}$. Then $G \leqslant G_B^B \Wr G^\mathcal{B}$ and the assumption that $\mathcal{B}$ is non-trivial ensures that both $G_B^B$ and $G^\mathcal{B}$ have strictly lower degree than $G$. Next we consider how Property $(*)_p$ interacts with these three groups.

\begin{lemma}\label{lem:imprimwreath}
Let $H$ act transitively on $\Omega$ and $K$ act transitively on $\Delta$. Then $H\Wr K$ acting imprimitively on $\Omega\times\Delta$ has Property~$(*)_p$ if and only if both $H$ and $K$ have Property~$(*)_p$.
\end{lemma}
\begin{proof}
Let $G=H\Wr K=H^k\rtimes K$ where $k=|\Delta|$ and let $\Sigma=\Omega\times\Delta$.  Then $R=P^k\rtimes Q$ is a Sylow $p$-subgroup of $G$, where $P$ is a Sylow $p$-subgroup of $H$ and $Q$ is a Sylow $p$-subgroup of $K$.  Let $\omega\in\Omega$ and $\delta\in\Delta$. Then $G_{(\omega,\delta)}=(H_{\omega}\times H^{k-1})\rtimes K_\delta$ and $R_{(\omega,\delta)}=(P_{\omega}\times P^{k-1})\rtimes Q_\delta$. 

If $H$ and $K$ both have Property~$(*)_p$, then $|P_\omega|$ does not depend on the choice of $\omega$ and $|Q_\delta|$ does not depend on the choice of $\delta$. Thus $|R_{(\omega,\delta)}|$ is constant and so Lemma \ref{lem:easy} implies that $G$ has Property~$(*)_p$.

Conversely, suppose that $G$ has Property~$(*)_p$. Then by Lemma \ref{lem:easy}, $|R_{(\omega,\delta)}|$ does not depend on $(\omega,\delta)$ and $|(\omega,\delta)^R|= |\Sigma|_p=|\Omega|_p|\Delta|_p$.
Now $|(\omega,\delta)^R|=|\omega^P||\delta^Q|$. By Lemma \ref{lem:shortest} we have that $|\omega^P|\geqslant |\Omega|_p$ and $|\delta^Q|\geqslant |\Delta|_p$. Thus, as $|(\omega,\delta)^R|=|\Omega|_p|\Delta|_p$ for all $\omega\in\Omega$ and $\delta\in\Delta$ it follows that $|\omega^P|= |\Omega|_p$ for all $\omega\in\Omega$ and $|\delta^Q|=|\Delta|_p$ for all $\delta\in\Delta$. Hence both $H$ and $K$ have Property~$(*)_p$.
\end{proof}

\begin{corollary}\label{cor:imprimwreath}
Let $G$ be an imprimitive group with system of imprimitivity $\mathcal{B}$. Let $B$ be a block of $\mathcal{B}$. Let $H=G_B^B$ be the permutation group induced on $B$ by the setwise stabiliser $G_B$, and let $K=G^{\mathcal{B}}$ be the group induced by $G$ on $\mathcal{B}$. If $H$ and $K$ have Property~$(*)_p$, then $G$ also has Property~$(*)_p$. 
\end{corollary}
\begin{proof}
The wreath product $H\Wr K$ has Property~$(*)_p$ by Lemma \ref{lem:imprimwreath}.
Since $G$ is a subgroup of $H\Wr K$, the statement follows immediately from Lemma \ref{lem:elementary}(5). 
\end{proof}

We have the following partial converse to Corollary \ref{cor:imprimwreath}.

\begin{lemma}\label{lem:GBB}
Suppose that $G$ acts transitively on $\Omega$ with Property $(*)_p$, that it is imprimitive with system of imprimitivity $\mathcal{B}$, and that $B\in\mathcal{B}$. Then $G_B^B$ has Property $(*)_p$.
\end{lemma}
\begin{proof}
Suppose that $G$ has Property $(*)_p$ and let $P\in\Syl_p(G_B)$. By Sylow's Theorems there exists a Sylow $p$-subgroup $Q$ of $G$ that contains $P$.  Thus $Q_B=P$. Suppose first that  $Q=P$.  Since $Q_\omega\in\Syl_p(G_\omega)$ for all $\omega\in \Omega$ and $(G_B)_\omega=G_\omega$, then $Q_\omega\in\Syl_p((G_B)_\omega)$ for all $\omega \in B$. Hence $G_B^B$ has Property $(*)_p$. Thus from now on suppose that $P<Q$.  Let $\alpha\in B$. Then  $Q$ acts imprimitively on $\alpha^Q$ with block $\alpha^Q \cap B=\alpha^P.$ Suppose that $B_1\in \mathcal{B}$ and there exists $q\in Q$ such that $B^q=B_1$. Then $\alpha^q\in B_1$ and so $B_1\cap\alpha^Q\neq \varnothing.$ Hence $|\alpha^Q|=|\alpha^P||B^Q|$.  Since $G$ has Property $(*)_p$ on $\Omega$, Lemma \ref{lem:easy} implies that $|\alpha^Q|$ is independent of the choice of $\alpha\in B$ and hence so is  $|\alpha^P|$. Thus by Lemma \ref{lem:easy}, $G_B$ has Property $(*)_p$ on $B$.
\end{proof}

We note that the full converse of Corollary \ref{cor:imprimwreath} does not hold.  For example, let $G=A_5$ acting on 20 points. Then $G$ has Property $(*)_2$, by Lemma~\ref{lem:elementary}(1), and $G$ acts imprimitively with a system of imprimitivity $\mathcal{B}$ consisting of 10 blocks of size 2.  However, $G^{\mathcal{B}}$ does not have Property $(*)_2$, by Lemma~\ref{lem:easy}, since a Sylow $2$-subgroup of $G$ has orbit lengths $2, 2, 2, 4$ in $\mathcal{B}$. We generalise this example below. Corollary \ref{cor:imprimwreath} and Lemma \ref{lem:GBB} then imply Theorem \ref{thm:imprim}.

\begin{lemma}\label{lem:imprimG^B}
Let $p$ be a prime, and let $G\leq \Sym(X)$ be a transitive permutation group of degree $d=kp+f$ with $0<f<p$. Let $P\in\Syl_p(G)$ and let $\fix_X(P)$ be the set of fixed points of $P$.  Assume that $|\fix_X(P)|=f$ and $P$ acts semiregularly on $X\setminus \fix_X(P)$.
Let $Y$ be an orbit of $G$ on the set of elements of $X^p$ with pairwise distinct entries. Then the natural action $G$ induces on $Y$ is faithful and imprimitive  with Property~$(*)_p$. Take $\mathcal{B}$ to be the partition of $Y$ where two $p$-tuples lie in the same block if they correspond to the same unordered $p$-set. If the blocks in this partition have size a multiple of $p$, then $G^{\mathcal{B}}$ does not have Property~$(*)_p$.
\end{lemma}

\begin{proof}
Since $P$ is a $p$-group, all its orbits on $X$ have size a power of $p$, so $P$ fixes at least $f$ points. We assume $P$ fixes exactly $f$ points. By definition $G$ acts transitively  on $Y$, and it is clear that $\mathcal{B}$ is a $G$-invariant partition of $Y$. Since $G$ is transitive on $X$, every element of $X$ occurs as the first element of a $p$-tuple in $Y$, so the subgroup of $G$ fixing each $p$-tuple in $Y$ fixes $X$ pointwise. Thus $G$ acts faithfully on $Y$.
To show that $G$ acts on $Y$ with Property~$(*)_p$, we need to show that all $P$-orbits on $Y$ have the same size. Consider the $p$-tuple $(x_1,x_2,\ldots,x_p)\in Y$. Since $|\fix_X(P)|=f<p$, there is an index $i$ such that $x_i\not\in \fix_X(P)$. Since $P$ acts semiregularly on $X\setminus \fix_X(P)$, $x_i^P$ has size $|P|$, and hence  $(x_1,x_2,\ldots,x_p)^P$ also has size $|P|$.  Therefore $G$ on $Y$ has Property~$(*)_p$ and $|P|=|Y|_p$ by  Lemma \ref{lem:easy}.

Let $K$ be the kernel of the action of $G$ on $\mathcal{B}$, so that the induced group $G^{\mathcal{B}}\cong G/K$. Now $G^{\mathcal{B}}$ is the natural action of $G$ on an orbit of unordered $p$-sets of $X$ (the subsets that correspond to  $p$-tuples in $Y$).
Let $Q=P^\mathcal{B}=PK/K$, so $Q\in\Syl_p(G^{\mathcal{B}})$. Then $Q\cong P/(P\cap K)$, and in particular $|Q|\leq |P|$.

Now assume each block of $\mathcal{B}$ has size $\ell p$ for some integer $\ell\geq 1$. 
 Note that $|\mathcal{B}|=|Y|/(\ell p)$, so $|\mathcal{B}|_p=|Y|_p/(p\,\ell_p)\leq |Y|_p/p= |P|/p$. If $G^{\mathcal{B}}$ had Property~$(*)_p$, then by Lemma \ref{lem:easy} all $Q$-orbits on ${\mathcal{B}}$ would have size $|\mathcal{B}|_p\leq |P|/p$. We will display an orbit of size $|P|$, yielding a contradiction (and also implying $P\cap K=1$).  
 Since $G$ is transitive on $X$ and since $f>0$, $Y$ contains a $p$-tuple $x=(x_1,x_2,\ldots,x_p)$ with  $x_p\in \fix_X(P)$. Further, since $f<p$, not all the $x_i$ are in $\fix_X(p)$.
 Consider the corresponding $p$-set $\widehat{x}:=\{x_1,x_2,\ldots,x_p\}$ where (without loss of generality) $x_1,x_2,\ldots,x_{p-a}$ are not fixed by $P$ and $x_{p-a+1},\ldots,x_{p}$ lie in $\fix_X(P)$, so $0<a\leq f<p$. 
Let $g \in P$  fix $\widehat{x}$ setwise.  Then $g$ fixes each of the points $x_{p-a+1},\ldots,x_{p}$ of $\fix_X(P)$, and also fixes $\{x_1,x_2,\ldots,x_{p-a}\}$ setwise. Since $g$ is a $p$-element and $0<p-a<p$ it follows that $g$ fixes 
each $x_i$ for $1\leq i\leq p$. Thus $g$ fixes the $p$-tuple $x=(x_1,\dots,x_p)\in Y$. We have already shown that each $P$-orbit in $Y$ has size $|P|$, and so we conclude that $g=1$. Thus the only element of $P$ fixing  $\widehat{x}$ setwise is the identity, so the $P$-orbit in $\mathcal{B}$ containing $\widehat{x}$ has size $|P|$. This completes the proof.
\end{proof}

If we take $G=A_5$ acting on a set $X$ of size $5$, then a Sylow $2$-subgroup $P$ is isomorphic to the Klein four-group, fixes one point $x$, and is semiregular on $X\setminus\{x\}$. Since $G$ is $2$-transitive on $X$, $G$ has a single orbit $Y$ on the set of ordered pairs with distinct entries from $X$. Thus the blocks in $\mathcal{B}$ have size $2$. All the conditions in Lemma \ref{lem:imprimG^B} hold, and this yields the example with $A_5$ of degree $20$ mentioned above. 

\begin{corollary}\label{cor:imprimegs}
Let $p$ be a prime.  There are infinitely many imprimitive groups $G$ that have Property~$(*)_p$ for which the induced action on some $G$-invariant block system does not have Property~$(*)_p$.
\end{corollary}

\begin{proof}
There are infinitely many $p$-powers $q$, and for each $q$  we will construct an example.
Take $G=\PGL_2(q)$, in its natural action of degree $d=q+1$ on $\mathbb{F}_q\cup\{\infty\}$. The subgroup $P$ of translations (with unique fixed point $\infty$) is a Sylow $p$-subgroup of $G$ and is semiregular on $\mathbb{F}_q$.
Let $Y$ be the orbit under $G$ of the $p$-tuple $x=(0,1,2,\ldots,p-1)$ (corresponding to the subfield $\mathbb{F}_p$ of  $\mathbb{F}_q$). 
By Lemma \ref{lem:imprimG^B}, $G$ has a natural faithful imprimitive action on $Y$ which has Property~$(*)_p$. Take $\mathcal{B}$ to be the partition of $Y$ where two $p$-tuples lie in the same block if they correspond to the same unordered $p$-set. The size of a block of $\mathcal{B}$ is $|G_{\{0,1,2,\ldots,p-1\}}|/|G_{(0,1,2,\ldots,p-1)}|$. 
Now $G_{\{0,1,2,\ldots,p-1\}}$ contains the group of translations by elements of $\mathbb{F}_p$, and hence is transitive on $\{0,1,2,\ldots,p-1\}$. Thus the size of a block is divisible by $p$, and hence
$G^{\mathcal{B}}$ does not have Property~$(*)_p$ by Lemma \ref{lem:imprimG^B}.
\end{proof}

Note these are not the only examples. We display here a second family satisfying the hypotheses of Lemma \ref{lem:imprimG^B} (this family has only $p-2$ examples for the prime $p$). Take $n=p+f$ where $p$ is a prime and $1<f<p$. Let $G=A_n$ acting naturally on $X$ of size $n$ and let $P\in\Syl_p(G)$. Then $|\fix_X(P)|=f$, and $P$ is cyclic of order $p$ and hence semiregular on $X\setminus\fix_X(P)$. 
Now $G$ is transitive on the set $Y$ of $p$-tuples of distinct points; for any such $p$-tuple there are $p!$ tuples in $Y$ corresponding to the same underlying subset $\beta$, and $G_\beta$ is transitive on these $p!$ tuples since $n\geq p+2$.
Thus the action of $G$ on $Y$ is imprimitive with a system of imprimitivity $\mathcal{B}$ consisting of $\binom{n}{p}$ blocks of size $p!$, so all the conditions of Lemma \ref{lem:imprimG^B} hold and it follows that $G^Y$ has Property~$(*)_p$, but $G^{\mathcal{B}}$ does not have Property~$(*)_p$.

We also note that the examples produced by Lemma \ref{lem:imprimG^B} have $G\cong G^{\mathcal{B}}$. There are also examples of groups $G$ having Property $(*)_p$ such that $G\not\cong G^{\mathcal{B}}$ and $G^{\mathcal{B}}$ does not have Property~$(*)_p$. For example, let $G=D_{12}$ acting regularly on itself. Then $G$ has Property~$(*)_2$. Since $G$ has a subgroup of order 4, this action has a system of imprimitivity $\mathcal{B}$ with  blocks of size 4. Then $G^{\mathcal{B}}=S_3$, which does not have Property $(*)_2$ by Lemma \ref{lem:elementary}(2).

\section{Primitive groups}
\label{sec:primitive}

By the O'Nan-Scott Theorem, see for example \cite{LPS}, the finite primitive permutation groups can be categorised into five types: (i) Affine type, (ii) Almost simple type, (iii) Diagonal type, (iv) Product action type, and (v) Twisted wreath type. For more information about these types see \cite{Cameron} or \cite{DM}.  We discuss the last three types in the next three sections.

\subsection{Diagonal type groups}
We show that primitive groups of Diagonal type do not have Property $(*)_p$.
\begin{lemma}\label{lem:diagonal}
Let $G$ be a primitive group of Diagonal type acting primitively on a set $\Omega$ and let $p$ be a prime dividing $|\Omega|$. Then $G$ does not have Property~$(*)_p$.
\end{lemma}
\begin{proof}
Suppose that $G$ has Property~$(*)_p$ and let $N=\soc(G)$. Then $N=T^k$ for some finite nonabelian simple group $T$ and integer $k\geqslant 2$. Let $D=\{(t,t,\ldots,t)\mid t\in T\}\leqslant N$. Then we may assume that $\Omega$ is the set of right cosets of $D$ in $N$. Letting $[t_1,t_2,\ldots,t_{k-1}]$ denote the coset $D(t_1,t_2,\ldots,t_{k-1},1)$ we have that $\Omega$ can be identified with $\{[t_1,t_2,\ldots,t_{k-1}]\mid t_1,\ldots,t_{k-1}\in T\}$. 

Let $p$ be a prime dividing $|\Omega|=|T|^{k-1}$ and so $p$ divides $|T|$.  Let $Q\in\Syl_p(T)$. Then $\widehat{Q}=Q^k$ is a Sylow $p$-subgroup of $N$ and by Sylow's Theorems, is contained in a Sylow $p$-subgroup $P$ of $G$. Now $\widehat{Q}=Q_1\times Q_2$ where $Q_1=Q^{k-1}\times 1$ and $Q_2=1^{k-1}\times Q$. For each $(g_1,\ldots,g_{k-1},1)\in Q_1$ and $t\in T$ we have that $[t,1,\ldots,1]^{(g_1,\ldots,g_{k-1},1)}=[tg_1,g_2,\ldots,g_{k-1}]$. Thus $[t,1,\ldots,1]^{Q_1}=\{[x_1,x_2,\ldots,x_{k-1}]\mid x_1\in tQ, x_2,\ldots,x_{k-1}\in Q\}$, which has size $|Q_1|=n_p$. By Lemma \ref{lem:elementary}(5), $N$ has Property~$(*)_p$ and so it follows from Lemma~\ref{lem:easy} that all orbits of $\widehat{Q}$ have size $n_p$, and hence that $[t,1,\ldots,1]^{\widehat{Q}}=[t,1,\ldots,1]^{Q_1}$. In particular, 
 $[t,1,\ldots,1]^{Q_2}\subseteq [t,1,\ldots,1]^{Q_1}$. Now $[t,1,\ldots,1]^{(1,1,\ldots,1,g)}=[g^{-1}t,g^{-1},\ldots,g^{-1}]$ and so $[t,1,\ldots,1]^{Q_2}=\{[x_1,x_2,x_2,\ldots,x_2]\mid x_1=x_2t,x_2\in Q\}$. Thus $Qt=tQ$ for all $t\in T$ and so $Q\norml T$, a contradiction.
\end{proof}

\subsection{Product action}

Let $k\geqslant 2$ be an integer and let $H\leqslant \Sym(\Delta)$ for a finite set $\Delta$. Then $\Sym(\Delta)\Wr S_k=\Sym(\Delta)^k\rtimes S_k$ acts on $\Omega=\Delta^k$ such that for each $(\delta_1,\delta_2,\ldots,\delta_k)\in\Omega$ we have 
$$(\delta_1,\delta_2,\ldots,\delta_k)^{(h_1,h_2,\ldots,h_k)}=(\delta_1^{h_1},\delta_2^{h_2},\ldots,\delta_k^{h_k}) \quad\quad \textrm{ for all $h_1,h_2,\ldots,h_k\in\Sym(\Delta)$}$$
and 
$$(\delta_1,\delta_2,\ldots,\delta_k)^\sigma=(\delta_{1\sigma^{-1}},\delta_{2\sigma^{-1}}  \ldots, \delta_{k\sigma^{-1}} )  \quad\quad \textrm{ for all $\sigma\in S_k$}.$$
Note that any  $G\leqslant \Sym(\Delta)\Wr S_k$ acts on $\{1,2,\ldots,k\}$ and the permutation group $G^{\{1,2,\ldots,k\}}$ induced is the projection of $G$ onto $S_k$. We denote the stabiliser in $G$ of 1 in this action by $G_1$. Note that we have a homomorphism from $G_1$ onto a subgroup of $\Sym(\Delta)$ given by
$$(h_1,h_2,\ldots,h_k)\sigma \mapsto h_1$$
and so the group $G_1$ acts naturally on $\Delta$.

We begin with a result about the full wreath product.

\begin{lemma}\label{lem:pcoprimefull}
Suppose that $p$ is a prime, $k$ is a positive integer  such that $k\geq 2$, and $K\leq S_k$ such that $|K|$ is coprime to $p$. Suppose also that $H$ acts transitively  on a set $\Delta$,  and
let $G=H\Wr K$ in product action on $\Omega=\Delta^k$. Then  $G$ has Property~$(*)_p$ on $\Omega$ if and only if  $H$ has Property~$(*)_p$ on $\Delta$.
\end{lemma}
\begin{proof}
Since $p$ does not divide $|K|$, the group $R=P^k$ is a Sylow $p$-subgroup of $G$, where $P$ is a Sylow $p$-subgroup of $H$. Let $\omega=(\delta_1,\ldots,\delta_k)\in\Omega$.  Then $R_\omega=P_{\delta_1}\times\cdots\times P_{\delta_k}$.
Suppose first that $H$ has Property~$(*)_p$. Then Lemma \ref{lem:easy} implies that $|R_{\delta_i}|$ does not depend on the choice of $\delta_i$ and so $|R_\omega|$ does not depend on $\omega$. Lemma \ref{lem:easy} then implies that $G$ has Property~$(*)_p$. Suppose conversely that  $G$ has Property~$(*)_p$. By  Lemma \ref{lem:easy}, $|R:R_\omega|=|\Delta|_p^k$ for all  $\omega\in\Omega$. Moreover, by Lemma \ref{lem:shortest} we have that $|P:P_{\delta_i}|\geqslant |\Delta|_p$ for all $\delta_i\in\Delta$. Thus $|P:P_{\delta_i}|= |\Delta|_p$ for all $\delta_i\in \Delta$ and so  again by  Lemma \ref{lem:easy}, $H$ has Property~$(*)_p$. 
\end{proof}

Throughout the remainder of this section we make the following assumption on the group $G$.

\begin{hypothesis}\label{productactionhyp}
Let $p$ be a prime, $k$ be a positive integer such that $k\geq 2$ and suppose that $G\leqslant \Sym(\Delta)\Wr S_k$ acts transitively on $\Omega=\Delta^k$ in product action. Let $K=G^{\{1,2,\ldots,k\}}$ and $H=G_1^{\Delta}$. Suppose that $H$ has a transitive subgroup $N$ such that $N^k\leqslant G\leqslant H\Wr K$. 
\end{hypothesis}

We first give a necessary condition for a group $G$ in product action to have Property~$(*)_p$.

\begin{lemma}\label{lem:reducePA}
Suppose that Hypothesis \ref{productactionhyp} holds. If $G$ has Property~$(*)_p$ on $\Omega$ then $H$ has Property~$(*)_p$ on $\Delta$.
\end{lemma}
\begin{proof}
Suppose that $G$ has Property~$(*)_p$ on $\Omega$ and that $H$ does not have Property~$(*)_p$ on $\Delta$. Then by Lemma \ref{lem:elementary}(5), $N^k$ has Property~$(*)_p$ on $\Omega$ and by Lemma \ref{lem:pcoprimefull}, $N$ has Property~$(*)_p$ on $\Delta$ (take $K=1$). Let $Q$ be a Sylow $p$-subgroup of $N$ and let $R=Q^k$. Then by Lemma \ref{lem:easy}, all orbits of $Q$ on $\Delta$ have size $|\Delta|_p$ and all orbits of $R$ on $\Omega$ have size $|\Omega|_p$.  Note that $R\leqslant G_1$.

Let $P$ be a Sylow $p$-subgroup of $H$. Since $G_1^\Delta=H$, there exists a Sylow $p$-subgroup $\widehat{P}$ of $G_1$ that contains $R$ and such that $\widehat{P}^\Delta=P$.  Since $G$ has Property~$(*)_p$ and $G_1$ is transitive on $\Omega$, Lemma \ref{lem:elementary}(5) implies that $G_1$ has Property~$(*)_p$ and so by Lemma \ref{lem:easy}  all orbits of $\widehat{P}$ have size $|\Omega|_p$.

As $H$ does not have Property~$(*)_p$ on $\Delta$,  there exists $\delta\in\Delta$ such that $|\delta^P|>|\Delta|_p$. Let $\omega=(\delta,\ldots,\delta)$. Then  $\omega^R=\delta^Q\times\delta^Q\times\cdots\times\delta^Q$ has size $|\Omega|_p$.
However, $\delta^Q\subset\delta^P$ and so $\omega^R\subset \omega^{\widehat{P}}$ as $\widehat{P}\leqslant G_1$ and induce $P$ on the first factor. This contradicts the fact that $G_1$ has Property~$(*)_p$ and so $H$ must have Property~$(*)_p$.
\end{proof}
 
 For sufficient conditions we need to split our analysis according to whether or not $p$ divides~$|K|$.
 
\begin{lemma}\label{lem:pcoprimeK}
Suppose that Hypothesis \ref{productactionhyp} holds and that $|K|$ is coprime to $p$.
 Then  $G$ has Property~$(*)_p$ on $\Omega$ if and only if  $H$ has Property~$(*)_p$ on $\Delta$.
\end{lemma}
\begin{proof}
Suppose first that $H$ has Property~$(*)_p$. Then by Lemma \ref{lem:pcoprimefull}, $H\Wr K$ has Property~$(*)_p$ and so by Lemma \ref{lem:elementary}(5) so does $G$. 

Conversely, suppose that $G$ has Property~$(*)_p$. We assume that $H$ does not have Property~$(*)_p$ on $\Delta$ and seek a contradiction. By Lemma \ref{lem:elementary}(5), $N^k$ has Property~$(*)_p$ on $\Omega$ and by Lemma \ref{lem:pcoprimefull}, $N$ has Property~$(*)_p$ on $\Delta$ (take $K=1$). Let $Q$ be a Sylow $p$-subgroup of $N$ and let $R=Q^k$. Then by Lemma \ref{lem:easy}, all orbits of $Q$ on $\Delta$ have size $|\Delta|_p$ and all orbits of $R$ on $\Omega$ have size $|\Omega|_p$.  Note that $R\leqslant G_1$.

Let $P$ be a Sylow $p$-subgroup of $H$. Since $G_1^\Delta=H$, there exists a Sylow $p$-subgroup $\widehat{P}$ of $G_1$ that contains $R$ and such that $\widehat{P}^\Delta=P$.  Since $G$ has Property~$(*)_p$ and $G_1$ is transitive on $\Omega$, Lemma \ref{lem:elementary}(5) implies that $G_1$ has Property~$(*)_p$ and so by Lemma \ref{lem:easy}  all orbits of $\widehat{P}$ have size $|\Omega|_p$.

As $H$ does not have Property~$(*)_p$ on $\Delta$,  there exists $\delta\in\Delta$ such that $|\delta^P|>|\Delta|_p$. Let $\omega=(\delta,\ldots,\delta)$. Then  $\omega^R=\delta^Q\times\delta^Q\times\cdots\times\delta^Q$ has size $|\Omega|_p$.
However, $\delta^Q\subset\delta^P$ and so $\omega^R\subset \omega^{\widehat{P}}$ as $\widehat{P}\leqslant G_1$ and induce $P$ on the first factor. This contradicts the fact that $G_1$ has Property~$(*)_p$ and so $H$ must have Property~$(*)_p$.
\end{proof}

Finally, we deal with the case where $p$ divides the order of $K$. 

\begin{lemma}\label{lem:pdivK}
Suppose that Hypothesis \ref{productactionhyp} holds and that $p$ divides $|K|$. Then  $G$ has Property~$(*)_p$ on $\Omega$ if and only if $|\Delta|$ is a power of $p$.
\end{lemma}

We recall from  Lemma~\ref{lem:elementary}(3) that,  if $|\Delta|$ is a power of $p$, then $H$ has Property~$(*)_p$ on $\Delta$.

\begin{proof}
If $|\Delta|$ is a power of $p$ then so is $|\Omega|$. Hence as $G$ acts transitively on $\Omega$ it follows from Lemma~\ref{lem:elementary}(3) that $G$ has Property~$(*)_p$.

Suppose conversely that $G$ has Property~$(*)_p$ on $\Omega$. Then by Lemma \ref{lem:elementary}(5) the transitive subgroup $N^k$ has Property~$(*)_p$ on $\Omega$ and hence by Lemma \ref{lem:pcoprimefull} the group $N$ has Property~$(*)_p$ on $\Delta$. Let $Q$ be a Sylow $p$-subgroup of $N$. Then by Lemma \ref{lem:easy} all orbits of $Q$ on $\Delta$ have size $|\Delta|_p$ and all orbits of $Q^k\in\Syl(N^k)$  on $\Omega$ have size $|\Omega|_p=|\Delta_p|^k$. Let $R$ be a Sylow $p$-subgroup of $G$ containing $Q^k$. We may assume that $R\leqslant P \Wr S$ where $P$ is a Sylow $p$-subgroup of $H$ containing $Q$ and $S$ is a Sylow $p$-subgroup of $K$. As $G$ has Property~$(*)_p$ all orbits of $R$ have size $|\Omega|_p=|\Delta|_p^k$.   Choose $\delta_1,\delta_2\in \Delta$ such that $\delta_1\neq \delta_2$, and consider $\omega':=(\delta_1,\delta_2,\delta_2,\ldots,\delta_2)\in\Omega$. Then $|(\omega')^R|=|\Delta|_p^k=|(\omega')^{Q^k}|$, and hence $Q^k$ acts transitively on $(\omega')^R$. Hence $R=Q^kR_{\omega'}$ and so $R_{\omega'}^{\{1,2,\ldots,k\}}=R^{\{1,2,\ldots,k\}}\leqslant S$. Since $p$ divides $|K|$ we have that $R^{\{1,2,\ldots,k\}}\neq 1$ and so we may assume that there exists $g=(h_1,h_2,\ldots,h_k)\sigma\in R_{\omega'}$ such that $1^\sigma=2$ and each $h_i\in P$. By the definition of the action, the image $(\omega')^g$ has second entry $\delta_1^{h_1}$, and since  $g$ fixes $\omega'$ it follows that $\delta_1^{h_1}=\delta_2$. Since $\delta_1, \delta_2$ were arbitrary distinct points of $\Delta$, we conclude that $P$ acts transitively on $\Delta$ and so $|\Delta|$ is a power of~$p$.
\end{proof}

Lemmas \ref{lem:pcoprimeK} and \ref{lem:pdivK} can be combined to get the following result.

\begin{theorem}\label{PA-p>k}
Let $G$ be a group satisfying Hypothesis \ref{productactionhyp}. 
\begin{enumerate}
    \item[(a)]  If $p$ does not divide $|K|$, then  $G$ has Property~$(*)_p$ on $\Omega$ if and only if  $H$ has Property~$(*)_p$ on $\Delta$.
    \item[(b)] If $p$ divides $|K|$, then  $G$ has Property~$(*)_p$ on $\Omega$ if and only if $|\Delta|$ is a power of $p$.
\end{enumerate}
\end{theorem}

\subsection{Twisted wreath type groups}

We show that primitive groups of Twisted wreath type do not have Property $(*)_p$.

\begin{lemma}\label{lem:TW}
Let $G$ be a primitive group of Twisted wreath type acting on $\Omega$ and let $p$ be a prime dividing $|\Omega|$. Then $G$ does not have Property~$(*)_p$.
\end{lemma}
\begin{proof}
Suppose that $G$ is primitive of twisted wreath type with socle $N=T^k$. Then $N$ acts regularly on $\Omega$. Thus we can identify $\Omega$ with $N$ and we can identify $G$ as the twisted wreath product $T\,\mathrm{twr}_\varphi \,P$ for some homomorphism $\varphi:Q\rightarrow\Aut(T)$ where $Q$ is a subgroup of $P$ of index $k$. We follow the treatment in  \cite[Section 6]{PS} and let $\omega\in\Omega$ correspond to $1_N$. Then $G_\omega=P$.
By \cite[Lemma 6.4]{PS} we have that that $T^k\leqslant G\leqslant H\Wr S_k$ where $H=G_1^{\Delta}\leqslant\Sym(\Delta)$ where $\Delta=T$. Now $N\norml G_1$ and so $H$ contains $T$ as a regular normal subgroup.  Now  $(G_1)_\omega=Q$ and $((G_1)_\omega)^\Delta=Q\varphi$. Thus $H\leqslant T\rtimes \Aut(T)$. By \cite[Corollary 6.3]{PS} we have that $\Inn(T)\leqslant\varphi(Q)$ and so $H=G_1^{\Delta}$ is a diagonal group with socle $T^2$ acting on $\Delta$. By Lemma \ref{lem:diagonal}, $T\times T$ does not have Property~$(*)_p$ on $\Delta$ and so by Lemma \ref{lem:elementary}(5), $H$ does not have Property~$(*)_p$ on $\Delta$. Hence Lemma \ref{lem:reducePA} implies that $G$ does not have Property~$(*)_p$ on $\Omega$.
\end{proof}

The following example shows that we cannot replace the primitive assumption in Lemma \ref{lem:TW} with quasiprimitive.

\begin{example}
Let $T$ be a finite nonabelian simple group with order divisible by a prime $p$ and let $\Delta=T$. Then $T$ acts regularly on $\Delta$  with Property~$(*)_p$. Let $k$ be a positive integer coprime to $p$ and let $G=T\Wr C_k$ acting in product action on $\Omega=T^k$. Then by Lemma \ref{lem:pcoprimefull}, $G$ has Property~$(*)_p$ on $\Omega$. Note that $N=T^k$ is the unique minimal normal subgroup of $G$ and acts regularly on $\Omega$. Thus $G$ is quasiprimitive of Twisted wreath type \cite[p156]{PS}.
\end{example}

\subsection{Proof of Theorem \ref{thm:primitive}}

Let $G$ be a primitive permutation group on $\Omega$ and suppose that $G$ has Property~$(*)_p$ for some prime $p$ dividing $|\Omega|$.  
By the  O'Nan-Scott Theorem $G$ has exactly one of the following five types: (i) Affine type, (ii) Almost simple type, (iii) Diagonal type, (iv) Product action type, and (v) Twisted wreath type. 

By Lemma \ref{lem:diagonal}, $G$ is not of Diagonal type and by Lemma \ref{lem:TW}, $G$ is not of Twisted wreath type. Moreover, if $G$ is of Almost simple type then case (1) of the theorem holds.

If $G$ is of Affine type then $G\leqslant \AGL(d,p)$ and $|\Omega|=p^d$. Thus by Lemma \ref{lem:elementary}(3), $G$ has Property~$(*)_p$ but does not have Property $(*)_r$ for any $r\neq p$ and $r$ dividing $|G|$. Hence case (2) of the theorem holds.

It remains to consider the case where $G$ is of Product action type. Then following the description in Case III(b) of \cite{LPS} we have that $\Omega=\Delta^k$ and there is a primitive group $H\leqslant\Sym(\Delta)$ of Almost simple or Diagonal type and with socle $N$, such that $N^k\leqslant G\leqslant H\Wr K$ for some $K\leqslant S_k$. Since $N$ is transitive on $\Delta$ it follows that $G$ satifies Hypothesis \ref{productactionhyp}. Hence by Lemma \ref{lem:reducePA}, $H$ has Property~$(*)_p$ on $\Delta$. Thus by Lemma \ref{lem:diagonal} $H$ is of Almost simple type. Moreover, by Lemma \ref{lem:pdivK}, if $p$ divides $|K|$ then $|\Delta|$ is a power of $p$. Thus case (3) of the theorem holds. Moreover, by Lemmas \ref{lem:pcoprimeK} and \ref{lem:pdivK} any group of satisfying these conditions  has Property~$(*)_p$.

\subsection{Almost-simple primitive groups of low degree}\label{sec:smalldegree}

Note that $\PSL_2(q)$ acts on the nonsingular conic of the Desarguesian projective plane of order $q$.
This is the natural 2-transitive action of $\PSL_2(q)$, but for $q \ne 7, 9$, it also has a (unique) primitive action of degree $\binom{q}{2}$, which is the action on the external lines to the nonsingular conic. In this action, for $q$ even, the point stabiliser is isomorphic to the dihedral group $D_{2(q+1)}$ of order $2(q+1)$. (Note: there are $q+1$ tangents to the conic, $\binom{q+1}{2}$ secants, and $q^2+q+1$ lines in total.)

\begin{lemma}\label{lem:eg}
The action of $\PSL_2(q)$, $q$ even, on $n=\binom{q}{2}$ points has Property $(*)_2$. Moreover, $2n_2<n$ for $q\geq 4$ and the stabiliser of a point has even order.
\end{lemma}

\begin{proof}
The stabiliser in $G:=\PSL_2(q)$ of an external line is isomorphic to the dihedral group $D_{2(q+1)}$.  Let $P$ be a Sylow $2$-subgroup of $G$. Then $P$ has order $q$. Indeed, $P$ fixes a unique point $X$ of the conic, and the stabiliser of $X$ in $G$ is of the form
$P\rtimes C_{q-1}$. Let $\ell$ be an external line. We claim that $|G_X\cap G_\ell|=2$, and hence
$|P\cap G_\ell|=2$. Let $t_X$ be the tangent line at $X$. Since $q$ is even, $t_X$ passes through the nucleus
$N$ of the conic. So $G_X\cap G_\ell$ fixes $X$, $N$, and the point $Y:=t_X\cap \ell$. Moreover, 
$G_X\cap G_\ell$ fixes every point of $t_X$ (because it fixes three points of $t_X$ and preserves cross-ratio). 
Hence a nontrivial element of $G_X\cap G_\ell$ is
the unique elation with axis $t_X$ and centre $Y$.
Therefore, the intersection of $P$ with the stabiliser of any external line has order 2, and so every $P$-orbit on 
external lines has size $q/2=n_2$. Since $n=n_2(q-1)$ and the stabiliser of a point is a dihedral group, the result follows.
\end{proof}

We remark that $\PGammaL_2(q)$, where $q=2^f$, acting on $\binom{q}{2}$ points, satisfies $(*)_2$ when
$f$ is odd as a Sylow 2-subgroup of $\PGammaL_2(q)$ is a Sylow 2-subgroup of $\PSL_2(q)$.

By a straight-forward computation in \textsf{GAP} \cite{GAP}, we have enumerated the
primitive groups  of Almost Simple type with Property $(*)_p$ of degree $n$ at most 4095 such that $pn_p<n$ and $p$ divides the order of the stabiliser of a point.  All such groups apart from when the socle is $\PSL_2(q)$, $q$ even, acting in degree $\binom{q}{2}$, are listed in Table \ref{tab:prim}. It is curious that the largest degree example known that does not have a socle isomorphic to $\PSL_2(q)$ has degree $135$. 

\begin{table}
\begin{tabular}{cccc}
\toprule
Degree & $G$ & $p$ & Rank\\
\midrule
6 & $\PSL_2(5)$ & 2 &  2\\ 
12 & $M_{11}$ & 3 &  2\\ 
36 & $\PSU_3(3)$ & 3 &  4\\ 
36 & $\PGammaU_3( 3)$ & 3 &  3\\ 
112 & $\PSU_4(3)$ & 2 &  3\\ 
135 & $\PSp_6( 2)$ & 3 &  4\\ 
\bottomrule\\
\end{tabular}
\caption{Some primitive groups $G$ of degree less than 4095 and with property $(*)_p$. }\label{tab:prim}
\end{table}

\begin{question}
Are there any Almost simple type primitive permutation groups of degree $n\geqslant 135$ with Property $(*)_p$ such that $pn_p<n$ and $p$ divides the order of the stabiliser of a point, other than those of degree $\binom{q}{2}$ and with socle $\PSL_2(q)$? \end{question}

\section{Proof of Theorem \ref{thm:2trans}}
\label{sec:primitive-proof}

     We note the following classic result which, together with Tornier's result \cite{tornier} for the alternating and symmetric groups, deals completely with the case where $n_p=p$.

\begin{theorem}\cite[Theorem]{Praeger}\label{thm:CP}
Let $p$ be a prime and $n$ a positive integer with $n_p=p$. Suppose that $G$ is a $2$-transitive permutation group on $n$ points that does not contain $A_n$ and let $P$ be a Sylow $p$-subgroup of $G$. If $G$ has Property~$(*)_p$ then one of the following holds:
\begin{enumerate}[{\rm (a)}]
    \item $p$ does not divide $|G_\omega|$;
    \item $|P|=4$ and $G=A_5$ with $n=6$;
    \item $|P|=9$ and $G=M_{11}$ with $n=12$;
\end{enumerate}
\end{theorem}

We need the following three preliminary results. The first concerns primitive prime divisors and can be found in 
\cite[Lemma 2.5]{guestp} (for $p=2$) and \cite[Lemma 4.1(iii)]{popiel}.

\begin{lemma}\label{lem:ppd}
Let $q$ be a prime power and let $p$ be a prime. Suppose that $p$ divides $q^m-1$, and let $e=o(q\ \text{mod}\ p)$, the least positive integer such that $p$ divides $q^e-1$. Then $p\equiv 1\pmod e$ and  $m=eb$, for some $b$, Moreover, 
\begin{enumerate}[(a)]
    \item if $p$ is odd then $(q^m-1)_p=(q^e-1)_p\cdot b_p$; and
    \item if $p=2$ then $e=1$ and $(q^m-1)_2=(q^2-1)_2\cdot (m/2)_2$ if $m$ is even, and $(q-1)_2$ if $m$ is odd.
\end{enumerate}
\end{lemma}

The second is a technical result about semilinear groups. 

\begin{lemma}\label{lem:field}
Let $p$ and $r$  be distinct primes, $f$ a positive integer divisible by $p$ and $d\geq 2$ a positive integer coprime to $p$ such that $p$ divides $r^{df}-1$. Let $\Sigma=\GF(r^{df})\backslash\{0\}$ and let $X=\langle \widehat{\xi},\phi\rangle$ be a Sylow $p$-subgroup of $\GammaL_1(r^{df})$ where $\widehat{\xi}:x\mapsto \xi x$ for some $\xi\in\GF(r^{df})$ of order $(r^{df}-1)_p$ and $\phi:x\mapsto x^{r^{df/f_p}}$. Let $Y=\langle \widehat{\xi}\rangle$ and $\sigma=\phi^{f_p/p}$. Then
\begin{enumerate}[(a)]
\item $p$ divides $r^{df/p}-1$.
    \item The group $X$ has exactly $p$ subgroups of order $p$ that intersect trivially with $Y$, and each is conjugate to $\langle\sigma\rangle$.
       \item $|C_Y(\sigma)|=(r^{df}-1)_p/p$.
        \item $\sigma$ fixes at most $r^{df/p}-1$ one-dimensional $\GF(r^f)$-subspaces of $\GF(r^{df})$.
     \item Suppose that $P\leqslant X$ with $Y<P$ and $P$ acts on a set $\Omega$ with all orbits having size $(r^{df}-1)_p$ and $Y$ acting semiregularly. Then $\langle Y,\sigma\rangle\leqslant P$ and $\sigma$ has $|\Omega|/p$ fixed points.
   
\end{enumerate}
\end{lemma}

\begin{proof}
Let $e=o(r\ \text{mod} \ p)$. Then by Lemma \ref{lem:ppd}, $p\equiv 1\pmod e$ and $e$ divides $df$. Hence $e$ divides $df/p$ and so $p$ divides $r^{df/p}-1$.

 The subgroups 
of $X$ of order $p$ that intersect trivially with $Y$ have the form 
    $\langle \widehat{\xi}^j\sigma\rangle$ for some integer $j$. Since $\sigma^{-1} \widehat{\xi}\sigma =\widehat{\xi}^{r^{df/p}}$ we have that $$
    (\widehat{\xi}^j\sigma)^p=\widehat{\xi}^{j(1+x+x^2+\cdots+x^{p-1})}=\widehat{\xi}^{j(x^p-1)/(x-1)}
    $$ where $x=r^{df(p-1)/p}$.  By Fermat's Little Theorem, $x^{p-1}\equiv 1\pmod p$ and so $p$ divides $x^p-1$ precisely when $p$ divides $x-1$. Moreover, in this case Lemma \ref{lem:ppd} implies that $((x^p-1)/(x-1))_p=p$. Thus if $\widehat{\xi}^j\sigma$ has order $p$ then $(r^{df}-1)_p/p$ divides $j$ and so there are at most $p$ subgroups of order $p$ of the form $\langle \widehat{\xi}^j\sigma\rangle$ in $X$. Moreover, by Lemma \ref{lem:ppd} and part (a), $(r^{df}-1)_p=(r^{df/p}-1)_pp$ and so $\xi^p\in\GF(r^{df/p})$, which is centralised by $\sigma$. Since $|\xi|$ does not divide $r^{df/p}-1$ it follows that $\xi\notin\GF(r^{df/p})$. Thus $C_X(\sigma)=\langle \widehat{\xi}^p,\phi\rangle$ and so $\langle\sigma\rangle$ has $p$ conjugates in $X$. Hence all subgroups of $X$ of order $p$ and that intersect trivially with $Y$ are conjugate in $X$ and parts (b) and (c) hold.
    
If $\sigma$ fixes the $\GF(r^f)$-subspace spanned by $\eta$ then $\eta^{r^{df/p}}=\lambda\eta$ for some $\lambda\in\GF(r^f)$. Hence $\eta^{(r^{df/p}-1)(r^f-1)}=1$ and so there are at most $(r^{df/p}-1)(r^f-1)$ such elements $\eta$. As each such one-dimensional subspace contains $r^f-1$ nonzero elements it follows that there are at most $r^{df/p}-1$ one-dimensional $\GF(r^f)$-subspaces fixed by $\sigma$. Thus (d) holds.

Suppose now that $P\leqslant X$ with $Y<P$ and $P$ acts on a set $\Omega$ with all orbits having size $(r^{df}-1)_p$ and $Y$ acting semiregularly. Since $|Y|=(r^{df}-1)_p$, this means that $Y$ is transitive (regular) on each of the $P$-orbits in $\Omega$. Then for all $\omega\in\Omega$ we have that $|P_\omega|= |P|/(r^{df}-1)_p=|P:Y|\geqslant p$. Hence $P_\omega$ contains a subgroup of order $p$ that intersects trivially with $Y$ and so by part (b) we may assume that $\langle Y,\sigma\rangle\leqslant P$ and for each orbit $O$ of $Y$ there exists $\omega\in O$ such that $\langle\sigma\rangle\leqslant P_\omega$. Moreover, the elements of $O$ can be identified with the elements of $Y$ such that $\omega$ corresponds to $1$ and $\sigma$ acts on $O$ as an automorphism of $Y$. Hence by (b), $\sigma $ fixes $(r^{df}-1)_p/p$ elements of $O$ and so fixes $|\Omega|/p$ elements of $\Omega$ as in part (e).
\end{proof}

The third preliminary result is the following number theoretic lemma.

\begin{lemma}\label{lem:morenum}
Suppose $p$ and $r$ are prime numbers, with $p$ odd, and let $f$ be an integer at least $3$, such that $p$ divides $f$ and $r^f+1$.
Then
\[
\frac{r^f+1}{p}\ge r^{2f/p}-1,
\]
with equality if and only if $(r,p,f)=(2,3,3)$.
\end{lemma}

\begin{proof}
The proof is by induction on $f$. If $f=3$, then $p= 3$ and, for all p[rimes $r\geq 2$,]
\[
\frac{r^f+1}{p}=\frac{r^3+1}{3} \ge r^{2}-1=r^{2f/p}-1.
\]
Suppose now that the result holds for some $f\ge 3$. Then
\begin{align*}
\frac{r^{f+1}+1}{p}&=r\cdot \frac{r^f+1}{p}+\frac{1}{p}(1-r)\\
&\ge r\cdot (r^{2f/p}-1)+\frac{1}{p}(1-r)\\
&=r^{2(f+1)/p}-1+r^{2f/p}(r-r^{2/p})+(1-r)(1+1/p).
\end{align*}
We will show that 
\[
r^{2f/p}(r-r^{2/p})+(1-r)(1+1/p)\ge 0.
\]
By elementary calculus, the function $x^2(x-x^{2/3})- \frac{4}{3}(x-1)$ is positive except for values of $x$ in the interval $[1,\varepsilon]$ where
$\varepsilon \approx 1.823$. So, since $r\ge 2$, we have $r^2(r-r^{2/3})\ge \frac{4}{3}(r-1)$ and hence
\[
r^{2f/p}(r-r^{2/p})\ge \frac{4}{3}(r-1)\ge (r-1)(1+1/p).
\]
Therefore, $r^{2f/p}(r-r^{2/p})+(1-r)(1+1/p)\ge 0$ and the inequality holds for $f+1$, and hence, by induction, for all $f\geq3$.

Finally, we consider when equality is attained: $\frac{r^f+1}{p}= r^{2f/p}-1$.
Let $q_0:=r^{f/p}$. Then $q_0^2-1$ divides $q_0^p+1$ and hence $q_0=2$. So $r=2$ and $f=p$. Therefore, $2^p+1=3p$.
By Fermat's little theorem $2^p\equiv 2\pmod{p}$ and so $0\equiv 3p \pmod{p}\equiv 2^p+1 \pmod{p}\equiv 2+1\pmod{p}$. Hence $p=3$.
\end{proof}

We now prove Theorem \ref{thm:2trans}.
\medskip

\noindent\emph{Proof of Theorem \ref{thm:2trans}.}
Let $p$ be a prime such that $G$ is a 2-transitive group with Property~$(*)_p$ of degree $n$. By Burnside's Theorem \cite[Theorem 4.1B]{DM}, $G$ is either affine or almost simple. If $G$ is affine then $n=r^k$ for some prime $r$, and by Lemma \ref{lem:elementary}(3) $G$ has Property~$(*)_p$ if and only if $p=r$. Thus $pn_p>n$ and so (a) holds. Hence it remains to consider the case where $G$ is almost simple. The almost simple 2-transitive permutation groups are listed in \cite[Table 7.4]{Cameron}.
If $\soc(G)=A_n$ then by Tornier \cite{tornier} we have that $pn_p>n$ and again case (a) holds. Also, if $n_p=p$ then one of (b), (c) or (d) holds by Theorem \ref{thm:CP}, so we also assume that $n_p\geq p^2$. 
We now work through the remaining possibilities for $\soc(G)$.

\medskip
\begin{enumerate}
\item \textsc{$\soc(G)=\PSL_d(q)$ with $n=\frac{q^d-1}{q-1}$ and $d\geq 2$:}

Suppose first that $\soc(G)=\PSL_d(q)$ and let $T=\soc(G)$. By Lemma~\ref{lem:elementary}(5), $T$ also has Property~$(*)_p$, so we will assume to start with that $G=T$. If $p\,n_p>n$ then case (a) holds (for both $G$ and $T$),  so we may assume that $n=n_pc$ with $c>p$ (since $c\ne p$ by the definition of $n_p$). (We will try not to use case (b) since it may hold for $T$ but not for $G$.)
In particular $n_p\geq p$ and $p$ divides $(q^d-1)/(q-1)$.  Let $e=o(q\ \text{mod}\ p)$, the least positive integer such that $p$ divides $q^e-1$. Let $(q^e-1)_p=p^a$. By Lemma~\ref{lem:ppd}, $d=eb$, for some $b$, and if $p$ is odd then $(q^{d}-1)_p=p^a\cdot b_p$, while if $p=2$ then $(q^{d}-1)_p$ is given by  Lemma~\ref{lem:ppd}(b). Note that, if $e>1$ then $p$ is odd, for if $p=2$ divides $q^d-1$ then $q$ is odd and so $e=1$.

\medskip
\begin{itemize}
    \item 
\textsc{Case both $b, e\geq2$.} \  As we just noted, $p$ is odd, and since in this case $p$ does not divide $q-1$, we will argue in $\GL_d(q)$ acting on the set of $1$-spaces of the underlying vector space $V=\mathbb{F}_q^d$. Then $P$ preserves a decomposition $V=\oplus_{i=1}^b U_i$ with $\dim(U_i)=e$ for each $i$, and $P$ contains $\prod_{i=1}^bP_i$ with $P_i= C_{p^a}$ a Sylow $p$-subgroup of $\GL(U_i)$. Now $P_i$ acts semiregularly on the nonzero vectors of $U_i$. Choose a nonzero $u_i\in U_i$ for each $i$, and set $u=\sum_{i=1}^bu_i$. Then the $P$-orbit on $1$-spaces containing $\langle u\rangle$ has size at least $p^{ab}$. On the other hand $n_p=(q^d-1)_p$ (since $e\geq 2$), and this equals $p^a.b_p$ which is strictly less than $p^{ab}$ since $p$ is odd,  a contradiction to Lemma~\ref{lem:easy}. 

\medskip
\item \textsc{Case $b\geq2$ but $e=1$.} Here $d=b$ and the pre-image $\widetilde{P}$ of $P$ in $\SL_d(q)$ has order $p^{a(d-1)}$ and acts on the set of $1$-spaces of $V$ with kernel of order $\gcd(d,(q-1)_p)=\gcd(d,p^a)$. There is a $\widetilde{P}$-invariant decomposition $V=\oplus_{i=1}^d U_i$ with this time each $\dim(U_i)=1$.  Choosing $u_i$ and $u$ as in the previous case we obtain an orbit of $\widetilde{P}$ on nonzero vectors of size at least $p^{a(d-1)}$, and hence a $P$-orbit on 1-spaces of length at least $p^{a(d-1)}/\gcd(d,p^a)\geq p^{a(d-2)}$. This time  $n_p=(q^d-1)_p/(q-1)_p=(q^d-1)_p\cdot p^{-a}$ (since $e= 1$). Since $n_p>1$ it follows from Lemma~\ref{lem:ppd} that 
$n_p=d_p$ if $p$ is odd, while if $p=2$ then $d$ is even and $n_2=(q+1)_2\cdot (d/2)_2$. Suppose first that either $p$ is odd, or $p=2$ and $q\equiv 1\pmod{4}$. Then $n_p=d_p$. If $d\geq 4$ or $d=3$ and $p\geq 3$ then $p^{a(d-2)}> d\geqslant d_p=n_p$, which is a contradiction to Lemma~\ref{lem:easy}. If $d=3$ and $p=2$ then $p^{a(d-1)}/\gcd(d,p^a)=2^{2a}>d_2=1$, another contradiction. Similarly, if $d=2$ and $p$ is odd then $p^{a(d-1)}/\gcd(d,p^a)=p^a>d_p=1$ while if $d=2$ and $q\equiv 1\pmod 8$ then $p^{a(d-1)}/\gcd(d,p^a)=2^{a-1}\geqslant 4>d_2=2$ and we again get a contradiction. When $d=p=2$ and $q\equiv 5\pmod 8$ we have that $n_2=2$, which contradicts the fact that $n_p\geq p^2$.

It remains to consider the case where $d$ is even, $p=2$ and $q\equiv 3\pmod 4$. Hence $n_2= (q+1)_2\cdot (d/2)_2 = \frac{1}{2}x\cdot (d/2)_2$, where $x=(q^2-1)_2$. For this case we observe that $\widetilde{P}$ contains a subgroup $Q$ of index 2 in  $\prod_{i=1}^{d/2}Q_i$ with each $Q_i\cong C_x$, and $Q$ leaves invariant a decomposition $V=\oplus_{i=1}^{d/2}W_i$ with each $W_i$ of dimension $2$.  The usual argument produces a $Q$-orbit on $1$-spaces of size at least $\frac{1}{4}x^{d/2}$. This  is strictly greater than $n_2$ if $d>2$, contradicting Lemma~\ref{lem:easy}, so we are left with $d=2$ and $q\equiv 3\pmod{4}$. In this last case $q=r^f$ for some odd prime $r$, and $f$ must also be odd. Hence $|\Out(T)|=2f$. If $|G/T|$ is odd, then $|G|_2=n_2$ so $|G_\omega|$ is odd and case (b) holds. So we may assume that $G$ contains $\PGL_2(q)$ and hence that $|P|\geqslant (q^2-1)_2=2n_2$. By Lemma~\ref{lem:elementary}(5), $H=\PGL_2(q)$ has Property $(*)_2$ also. As we are assuming that $n>n_2$ we have that $n=q+1$ is not a power of 2. Now a Sylow $2$-subgroup of $\PGL_2(q)$ is a dihedral subgroup of $D_{2(q+1)}$, and has at least one orbit of each length $n_2$ and $2\,n_2$, contradicting Lemma~\ref{lem:easy}.

\medskip
\item \textsc{Case $b=1$ and so $e=d$.}
The remaining case is $b=1, e=d$, so $p$ is a primitive prime divisor of $q^d-1$. Since $p$ does not divide $q-1$ we again argue in $\GammaL_d(q)$.  Let $q=r^f$ with $r$ prime and $f\geq 1$. Then the $p$-part of $|\Out(T)|$ is $f_p$. If $p$ does not divide $|G/T|$ then $|G_\omega|$ is not divisible by $p$ and we have case (b). So we may assume that $f_p>1$ and that $p$ divides $|G/T|$. Moreover, we may assume that $(r,f,d)\neq (2,3,2)$ as in this case $n=9$ and hence $p=3$ and $n=p^2$, so we are in case (a). Furthermore, since $p\equiv 1\pmod e$ we have that $p$ is coprime to $d$. Thus $p,r,f$ and $d$ satisfy the hypotheses of Lemma \ref{lem:field} and we consider $V$ as $\GF(q^d)$. Moreover, since $G$ has Property~$(*)_p$ we may assume that $P\leqslant\GammaL_1(r^{df})$ and satisfies the hypotheses on the group $P$ in Lemma \ref{lem:field}(e) with $Y=P\cap\GL_d(q)$. In particular, $\sigma$, as defined in the statement of Lemma \ref{lem:field}, lies in $P$ and has $n/p$ fixed points on $\Omega$. However, by Lemma \ref{lem:field}(d), $\sigma$ has at most $r^{df/p}-1$ fixed points on $\Omega$. If $d=2$ then $n=r^f+1$ and so this contradicts Lemma \ref{lem:morenum}. Thus $d\geq 3$.  However, since $p>d$ it follows that $r^{df/p}-1<q$ and since $p$ divides $r^{df/p}-1$ (by Lemma \ref{lem:field}(a)) we have that $p<q$. Hence $n=q^{d-1}+q^{d-2}+\cdots +q+1>pq$, and so $n/p> q> r^{df/p}-1$. Thus we again get a contradiction to the fact that $\sigma$ has $n/p$ fixed points.

\end{itemize}

\medskip

\item \textsc{$\soc(G)=\PSU_3(q)$ with $n=q^3+1$:}

Suppose that $\soc(G)=\PSU_3(q)$ and let $T=\soc(G)=\PSU_3(q)$. Then $G\leqslant \PGammaU_3(q)$ with $q=r^f$. Let $V$ be a 3-dimensional vector space over $\GF(q^2)$ equipped with a $G$-invariant nondegenerate hermitian form $B$. Moreover,  $\Omega$ is the set of totally isotropic $1$-subspaces of $V$ with respect to the form $B$. Now $|\PSU_3(q)|=\frac{1}{(3,q+1)} q^3(q^2-1)(q^3+1)$  and $\PGammaU_3(q)=q^3(q^2-1)(q^3+1)2f$. Suppose that we are not in case (b) and so $p$ divides $|G_\omega|$. Then $p$ divides $(q^2-1)2f$. Note that if $p=2$ then $q$ is odd and $p$ divides $q^2-1$. Thus we have two cases to consider: $p$ divides $q^2-1$ or $p$ divides $f$.

\begin{itemize}
\item \textsc{Case $p$ divides $q^2-1$.} Now \begin{align*}
\gcd(q^3+1,q^2-1)&=\gcd( (q+1)(q^2-q+1), q^2-1) )\\
&=\gcd( (q+1)(q^2-1)+(q+1)(-q+2), q^2-1 )\\
&=(q+1)\gcd( -q+2, q-1 )\\
&=q+1
\end{align*}
and so $p$ divides $q+1$. Thus by \cite[Lemma A.4]{BGbook} we have $$n_p=\left\{\begin{array}{ll}
                             (q+1)_p  &  \textrm{ if $p\neq 3$}\\
                             3(q+1)_3   &  \textrm{ if $p=3$.}
                                \end{array}\right.$$
                                
Let $v_1,v_2,v_3$ be an orthonormal basis for $V$ and consider the subgroup $X$ of $\PGU_3(q)$ whose preimage in $\GU_3(q)$ is
$$\widetilde{X}=\left\{\begin{pmatrix} \lambda_1 &0&0\\0&\lambda_2 &0\\ 0&0 &\lambda_3 \end{pmatrix} \mid \lambda_i^{q+1}=1\right\}.$$
 Then $|X|=(q+1)^2$. Consider the set $\Sigma$ of 1-dimensional subspaces of $V$ of the form $\langle a_1v_1+a_2v_2+v_3\rangle$ for $a_1,a_2\in\GF(q^2)$. If $$g=\begin{pmatrix} \lambda_1 &0&0\\0&\lambda_2 &0\\ 0&0 &\lambda_3 \end{pmatrix} \in\widetilde{X}$$ then $\langle a_1v_1+a_2v_2+v_3\rangle^g=\langle \lambda_1\lambda_3^{-1}a_1v_1+\lambda_2\lambda_3^{-1}a_2v_2+v_3\rangle$. Thus $g$ fixes an element of $\Sigma$ precisely if $\lambda_1=\lambda_2=\lambda_3$, in which case $g$ corresponds to the identity element of $X$. Hence $X$ acts semiregularly on $\Sigma$.
 Now $B(a_1v_1+a_2v_2+v_3,a_1v_1+a_2v_2+v_3)=1+a_1^{q+1}+a_2^{q+1}$. Since $-1-a_2^{q+1}\in\GF(q)$ for all $a_2\in\GF(q^2)$, it follows that for any $a_2$ there exists $a_1\in\GF(q^2)$ such that $a_1^{q+1}=-1-a_2^{q+1}$ and so $\Sigma$ contains totally isotropic 1-subspaces. In particular, $X$ has a regular orbit on $\Omega$.

Let $Q$ be the unique Sylow $p$-subgroup of $X\cap T$ and let $P$ be a Sylow $p$-subgroup of $G$ containing $Q$. Note that the preimage of $X\cap T$ in $\widetilde{X}$ consists of all elements of $\widetilde{X}$ such that $\lambda_1\lambda_2\lambda_3=1$. Suppose first that $p\neq 3$. Then $|Q|=(q+1)_p^2$. Since $X$ has a regular orbit on $\Omega$ it follows that $P$ has an orbit of size at least $((q+1)_p)^2>n_p=(q+1)_p$, contradicting Property~$(*)_p$. Thus $p=3$. Then $|Q|=(q+1)_3^2/3$.  Let $\sigma\in\PSU_3(q)$ be the element of order 3 whose preimage in $\SU_3(q)$  cyclically permutes $v_1,v_2$ and $v_3$. Then $\sigma$ normalises $X$ and $Q$. In particular, $\widehat{Q}=\langle Q,\sigma\rangle$ is a 3-subgroup of $T$ of order $(q+1)_3^2$, and we may assume that $\widehat{Q}\leqslant P$. As discussed in the previous paragraph, there exists $a_1\in\GF(q^2)$ such that $a_1^{q+1}=-2$ and so $\omega=\langle a_1v_1+v_2+v_3\rangle$ is a totally singular subspace in $\Sigma$. Now if $\langle b_1v_1+b_2v_2+v_3\rangle\in \omega^{Q}$ then $b_2^{q+1}=1$.
However, $\omega^\sigma=\langle v_1+a_1v_2+v_3\rangle$ and $a_1^{q+1}=-2\neq 1$. Hence $\omega^Q\subset \omega^{\widehat{Q}}$ and so $|\omega^{\widehat{Q}}|=(q+1)_3^2$. If $(q+1)_3>3$ then  it follows that $P$ has an orbit of length greater than $n_3=3(q+1)_3$, contradicting $G$ having Property $(*)_3$. Thus $(q+1)_3=3$. Since $3$ divides $|G_\omega|$ it follows that either $\PGU_3(q)\leqslant G$ or $3$ divides $f$. In the first case, taking $Q$ to be the unique Sylow $3$-subgroup of $X$ and $\widehat{Q}=\langle Q,\sigma\rangle$, we again see that $|\omega^{\widehat{Q}}|=|\widehat{Q}|=(q+1)_3^23>n_3$, contradicting Property $(*)_3$.  Finally, if $3$ divides $f$ and $q+1=r^f+1$, note that $3$ does not divide $r^f-1$ and so does not divide $r-1$. Hence $3$ divides $r+1$ and so by \cite[Lemma A.4]{BGbook}, $(q+1)_3=(r+1)_3f_3>3$, contradicting our assumption that $(q+1)_3=3$. 

\item \textsc{Case $p$ divides $f$ but not $q^2-1$.} Recall that $p$ divides $n=q^3+1$. Also, in this case $p\neq 3$ as $(3,q)=1$ and so 3 divides either $q-1$ or $q+1$. In particular, a Sylow $p$-subgroup of $\PGU_3(q)$ is contained in $\PSU_3(q)$. Moreover, as $p$ divides $|G_\omega|$ we have that $p$ divides $|G:\PSU_3(q)|$. Let $V=\GF(q^6)$, treated as a vector space over $\GF(q^2)$. Then the map $B(v,w)=\mathrm{Tr}_{q^6\rightarrow q^2}(vw^{q^3})$ is a nondegenerate hermitian form on $V$ and the totally singular vectors are all the elements $v\in \GF(q^6)$ such that $\mathrm{Tr}_{q^6\rightarrow q^2}(v^{q^3+1})=0$. Letting $d=6$ we have that $p,r,f$ and $d$ satisfy the hypotheses of Lemma \ref{lem:field}. Moreover, as $G$ satisfies Property~$(*)_p$ we have that $P$ is a subgroup of $\GammaL_1(r^{6f})$ and satisfies the hypotheses of part Lemma \ref{lem:field}(e) with $Y$ being the Sylow $p$-subgroup of $\PSU_3(q)$. In particular, $\sigma$, as defined in the statement of Lemma \ref{lem:field}, lies in $P$ and has at most $n/p$ fixed points on $\Omega$. However, $p<q$ and so $n/p=(q^3+1)/p>q$ while by Lemma \ref{lem:field}(d) we have that $\sigma$ fixes at most $r^{3f/p}-1<q$ elements of $\Omega$, a contradiction.
\end{itemize}

\medskip

\item \textsc{$\soc(G)=\Ree(q)$ with $n=q^3+1$:} Suppose that $\soc(G)=\Ree(q)$ and let $T=\Ree(q)$. In this case we have $q=3^f$ for some odd integer $f\geq 3$. Now $|T|=q^3(q^3+1)(q-1)$ and $|\Out(T)|=f$. Now $\gcd(q^3+1,q-1)=2$. Thus if (b) does not hold then there are two cases to consider: $p=2$, or $p$ divides $f$ and $|G:T|$ is divisible by $p$. Note that $p\neq 3$. Let $P$ be a Sylow $p$-subgroup of $G$.

\medskip
\begin{itemize}
    \item \textsc{Case $p=2$:} Note that since $q=3^f$ for $f$ odd, \cite[Lemma A.4]{BGbook} implies that $(q^3+1)_2=(q+1)_2=4$ and $(q-1)_2=2$. Thus $P\leqslant T$ and $|P|=8$. Now $P$ is elementary abelian, all involutions in $G$ are conjugate and the centraliser of an involution in $T$ is $C_2\times \PSL_2(q)$, see \cite{KlG2}. Thus $T$ and hence $G$ contains $q^2(q^2-q+1)$ involutions. Moreover, given a point $\omega\in\Omega$, we have that $T_\omega=[q^3]:C_{q-1}$ where $[q^3]$ denotes  a subgroup of order $q^3$. Since $q^2$ does not divide the centraliser of an involution, it follows that $T_\omega$ and hence $G_\omega$ contains $q^2$ involutions. Hence double counting shows that each involution in $G$ fixes $q+1$ points. Let $O$ be an orbit of $P$ of length 4. Then there is a unique involution in $P$ that fixes $O$ pointwise, and the remaining involutions are derangements on $O$. Since $P$ contains 7 involutions it follows that $P$ has at most $7(q+1)/4$ orbits of size 4. However, for $q>3$ we see that $n/4> 7(q+1)/4$ and so by Lemma \ref{lem:easy}, $G$ does not have Property $(*)_2$, a contradiction. It is interesting to note that in the excluded case when $q=3$ we have that $G=\Ree(3)=\PGammaL_2(8)$ in its action on 28 points. Here $n/4=7$ and $P$ indeed has 7 orbits of length 4 and $G$ has Property $(*)_2$, as in part (e).

\medskip
\item \textsc{Case $p$ divides $f$:} In this case we have $p\geq 5$ and $p$ is a primitive prime divisor of $q^2-1$ or $q^6-1$. In both cases, $|\Aut(T)|_p=n_pf_p$ and a Sylow $p$-subgroup of $T$ acts semiregularly on $\Omega$. Now $\Ree(q)\leqslant\GL_6(q)$ and $\Omega$ can be identified with a set of 1-dimensional subspaces of $V=\GF(q)^6$ of size $q^3+1$.

Suppose first that $p$ is a primitive prime divisor of $q^6-1$. Then $p\equiv 1\pmod 6$ and so $p\geq 7$. Letting $d=6$ and $r=3$ we have that $p,r,f$ and $d$ satisfy the hypotheses of Lemma \ref{lem:field}.   As $G$ satisfies Property~$(*)_p$ we have that $P$ is a subgroup of $\GammaL_1(r^{6f})$ that satisfies the hypotheses of part Lemma \ref{lem:field}(e) with $Y$ being the Sylow $p$-subgroup of $\Ree(q)$. In particular, $\sigma$, as defined in the statement of Lemma \ref{lem:field} lies in $P$ and has at most $n/p$ fixed points on $\Omega$. However, $p<q$ and so $n/p=(q^3+1)/p>q$ while by Lemma \ref{lem:field}(d) we have that $\sigma$ fixes at most $r^{3f/p}-1<q$ elements of $\Omega$, a contradiction. Thus $p$ is a primitive prime divisor of $q^2-1$. 

By \cite{KlG2}, $T$ has a maximal subgroup $H=C_2\times\PSL_2(q)$ which contains a Sylow $p$-subgroup of $T$. Moreover, as the centraliser of a field automorphism of order $p$ is $\Ree(3^{f/p})$, field automorphisms of $\Ree(q)$ of order a power of $p$ normalising $H$ induce field automorphisms on $\PSL_2(q)$.  Hence letting $d=2$ and $r=3$, we can view $P$ as a subgroup of $\GammaL_1(r^{2f})$ that satisfies the hypotheses of part Lemma \ref{lem:field}(e) with $Y$ being a Sylow $p$-subgroup of $\Ree(q)$ contained in $H$. Hence $\sigma$, as defined in the statement of Lemma \ref{lem:field}, lies in $P$ and has at most $n/p$ fixed points on $\Omega$.  Note that $n/p=(q^3+1)/p>q$ as $p<q$. By \cite[Proposition 4.91(d)]{GLS}, subgroups of $\Aut(T)$ of order $p$ and which intersect $T$ trivially are all conjugate. Hence the number of fixed points of $\sigma$ on $\Omega$ is the same as the number of fixed points as the element $\sigma$ in the previous paragraph. However, this number is at most $r^{3f/p}-1<q$, yielding a final contradiction.
\end{itemize}

\medskip

\item \textsc{$\soc(G)=\Sz(q)$ with $n=q^2+1$:} Suppose that $\soc(G)=\Sz(q)$ and let $T=\Sz(q)$. In this case we have $q=2^f$ for some odd integer $f\geq 3$. If $p$ divides $n$ then $p$ divides $q^4-1$. Hence $p$ is a primitive prime divisor of $q-1$, $q^2-1$ or $q^4-1$.
  Since $q$ is even, it is not possible for $p$ to divide $q^2+1$ and either $q-1$ or $q^2-1$. Hence $p$ is a primitive prime divisor of $q^{4}-1$. Now $|G_\omega|$ divides $q^2(q-1)f$. Thus either (b) holds, or $p$ divides $f$ and  $|G:T|$ is divisible by $p$. We consider the latter case. Letting $d=4$ and $r=2$ then $d$, $r$, $f$ and $p$ satisfy the hypotheses of Lemma \ref{lem:field}. 
   Now $\Sz(q)\leqslant \Sp_4(q)$ and $\Omega$ can be identified with a set of 1-dimensional subspaces of $V=\GF(q)^4$ of size $q^2+1$.  As $G$ satisfies Property~$(*)_p$ we have that $P$ is a subgroup of $\GammaL_1(r^{4f})$ that satisfies the hypotheses of part Lemma \ref{lem:field}(e) with $Y$ being the Sylow $p$-subgroup of $\Sz(q)$. In particular, $\sigma$, as defined in the statement of Lemma \ref{lem:field}, lies in $P$ and has at most $n/p$ fixed points on $\Omega$. However, $p<q$ and so $n/p=(q^2+1)/p>q$ while by Lemma \ref{lem:field}(d) we have that $\sigma$ fixes at most $r^{2f/p}-1<q$ elements of $\Omega$, a contradiction.

\medskip

\item \textsc{$\soc(G)=\Sp_{2d}(2)$ with $n=2^{2d-1}\pm 2^{d-1}$ and $d\geq 3$:} Suppose that $\soc(G)=\Sp_{2d}(2)$. Then $G=\soc(G)$.
In this case, for $\epsilon\in\{\pm\}$ there are sets $\Omega^\epsilon$ with $|\Omega^\epsilon|=2^{2d-1}+ \epsilon 2^{d-1}$ such that $G$ acts 2-transitively on each $\Omega^\epsilon$.  

Let $V$ be a ${2d}$-dimensional vector space over $\GF(2)$ and let $B$ be a nondegenerate alternating form on $V$ preserved by $G$. We say that a quadratic form $Q:V\rightarrow \GF(2)$ \emph{polarises} to $B$ if $B(v,w)=Q(v)+Q(w)-Q(v+w)$ for all $v,w\in V$. Then $G$ acts on the set of all such quadratic forms via $Q^g(v)=Q(v^{g^{-1}})$ for all $v\in V, g\in G$. Moreover, $G$ has two orbits, namely $\Omega^+$ and $\Omega^-$ where $\Omega^+$ is the set of all hyperbolic quadratic forms polarising to $B$ and $\Omega^-$ is the set of all elliptic quadratic forms polarising to $B$. In particular, if $Q\in\Omega^+$ then $G_Q\cong\GO^+_{2d}(2)$ while if $Q\in\Omega^+$ then $G_Q\cong\GO^-_{2d}(2)$.

Let $g\in G$ and let $Q\in \Omega^+\cup\Omega^-$ such that $Q^g=Q$. Then following \cite{Dye} and noting that since $x^2=x$ for all $x\in\GF(2)$, the elements of  $\Omega^+\cup\Omega^-$ are precisely the quadratic forms $Q_a$, for $a\in V$ given by
$$Q_a(v)=Q(v)+B(a,v).$$
Moreover, $(Q_a)^g(v)=Q_a(v^{g^{-1}})=Q(v^{g^{-1}})+B(a,v^{g^{-1}})=Q(v)+B(a^g,v)$ as $g$ preserves $Q$ and $B$.
Now $\GO^\epsilon_{2d}(2)$ has two orbits on the nonzero vectors of $V$ and these have size $2^{d-1}(2^d-\epsilon)$ (the nonsingular vectors) and $(2^d-\epsilon)(2^{d-1}+\epsilon)$ (the singular vectors). Suppose that $\GO^\epsilon_{2d}(2)$ fixes $Q$. Then the action of $\GO^\epsilon_{2d}(2)$ on $V$ is equivalent to the action of $\GO^\epsilon_{2d}(2)$ on $\Omega^+\cup\Omega^-$. Thus the quadratic forms $Q_a$ lying in the same $G$-orbit as $Q$ are those of the form $Q_a$ such that $Q(a)=0$. In particular, if $H\leqslant\GO^\epsilon_{2d}(2)\leqslant G$, then the orbits of $H$ on the set of singular vectors of $Q$ correspond to the orbits of $H$ on $\Omega^\epsilon\backslash\{Q\}$. Let $e_1,f_1\in V$ such that $Q(e_1)=Q(f_1)=0$ and $B(e_1,f_1)=1$. By \cite[4.1.20]{KL}, $\GO^\epsilon(2)_{e_1}=\widehat{P}\rtimes \GO^\epsilon_{2d-2}(2)$ where $|\widehat{P}|=2^{2d-2}$ and acts trivially on $\langle e_1,f_1\rangle^\perp$. Thus $\widehat{P}_{f_1}$ acts trivially on $V$ and so $|f_1^{\widehat{P}}|=|\widehat{P}|=2^{2d-2}$. Since $Q(f_1)=0$ we have that $Q_{f_1}\in\Omega^\epsilon$ and so $\widehat{P}$ has an orbit of length $2^{2d-2}>n_2=2^d$ on $\Omega^\epsilon$. (Recall that $d\geq 3$.) This contradicts Lemma \ref{lem:shortest} and so $p\neq 2$. 

To deal with the $p$ odd case we need to set up some more notation. Let $V=V_1\perp V_2\perp \ldots\perp V_k$ be an orthogonal decomposition of $V$ such that $B_i=B\mid_{V_i\times V_i}$ is a nondegenerate alternating form. Let $Q_i:V_i\rightarrow \GF(2)$ be a nondegenerate quadratic form on $V_i$ that polarises to $B_i$ such that $\mathrm{sign}(Q_i)=\epsilon_i$. For each $v\in V$ write $v=v_1+v_2+\cdots+v_k$ where each $v_i\in V_i$ and define $(Q_1\oplus Q_2\oplus\cdots\oplus Q_k)(v)=Q_1(v_1)+Q_1(v_2)+\cdots+Q_1(v_k)$. Then $Q=Q_1\oplus Q_2\oplus\cdots \oplus Q_k$ is a nondegenerate quadratic form that polarises to $B$ and $Q\in\Omega^\epsilon$ for $\epsilon=\Pi_{i=1}^k \epsilon_i$.  Moreover, for $g_i\in\Sp(V_i)$ we have that $g=(g_1,g_2,\ldots,g_k)\in G$ and $Q^g=Q_1^{g_1}\oplus Q_2^{g_2}\oplus\cdots\oplus Q_k^{g_k}\in\Omega^\epsilon$. See \cite{Mark} for more details.

Suppose first that $G$ has Property~$(*)_p$ on $\Omega^+$ with $p$ odd. Since $n=2^{d-1}(2^d+1)$ we have that  $p$ divides $2^d+1$ and  hence divides $2^{2d}-1$. Let $e=o(2\pmod p)$ and let $(q^e-1)_p=p^a$. Then by Lemma \ref{lem:ppd}, $2d=eb$ and $(q^{2d}-1)_p=p^ab_p$. If $e=2d$ then $n_p=|G_p|$ and so case (b) holds. Thus suppose that $e<2d$. If $e$ divides $d$ then $p$ divides $2^d-1$, which is coprime to $n$. Thus $e$ is even. Let $V=V_1\perp V_2\perp\ldots\perp V_b$ where each $V_i$ is a nondegenerate subspace of dimension $e$ and let $B_i$ be the restriction of $B$ to $V_i\times V_i$. Then $\Sp(V_i)$ acts transitively on the set $\Omega_i^+$ of $2^{e/2-1}(2^{e/2}+1)$ hyperbolic quadratic forms on $V_i$ that polarise to $B_i$. Let $P_i\in\Syl_p(\Sp(V_i))$. By Lemma \ref{lem:shortest} there exists a hyperbolic quadratic form $Q_i\in\Omega_i^+$ such that $|Q_i^{P_i}|=p^a$. Then $Q:=Q_1\oplus Q_2\oplus \cdots\oplus Q_b\in\Omega^+$. Moreover, letting $\widehat{P}=\Pi_{i=1}^b P_i$ we see that $|Q^{\widehat{P}}|=p^{ab}>n_p=p^ab_p$. Thus $G$ has a Sylow $p$-subgroup with an orbit of length greater than $n_p$,  contradicting Lemma \ref{lem:easy}.  Thus $G$ does not have Property~$(*)_p$ on $\Omega^+$.

Next suppose that $G$ has Property~$(*)_p$ on $\Omega^-$ with $p$ odd. Then $n=2^{d-1}(2^d-1)$ and so $p$ divides $2^d-1$.  Let $e=o(2\pmod p)$ and let $(q^e-1)_p=p^a$. Then by Lemma \ref{lem:ppd}, $d=eb$ and $(q^{d}-1)_p=p^ab_p$. Recall that $\dim(V)=2d=2eb$ in this case. Suppose first that  $e=d$ and $d$ is odd.  Then $n_p=|G|_p$ and so case (b) holds. Next suppose that $e$ is even. Let $V=V_1\perp V_2\perp\ldots\perp V_{2b}$ where each $V_i$ is a nondegenerate subspace of dimension $e$ and let $B_i$ be the restriction of $B$ to $V_i\times V_i$.  For $i=1,\ldots, 2b-1$, $\Sp(V_i)$ acts transitively on the set $\Omega_i^+$ of $2^{e/2-1}(2^{e/2}+1)$ hyperbolic quadratic forms on $V_i$ that polarise to $B_i$. Note that $(2^{e/2}+1)_p=p^a$. Let $P_i\in\Syl_p(\Sp(V_i))$. Then by Lemma \ref{lem:shortest} there exists a hyperbolic quadratic form $Q_i\in\Omega_i^+$ such that $|Q_i^{P_i}|=p^a$. Finally, note that $\Sp(V_{2b})$ acts transitively on the set $\Omega_{2b}^-$ of $2^{e/2-1}(2^{e/2}-1)$ hyperbolic quadratic forms on $V_{2b}$ that polarise to $B_{2b}$.  Let $P_{2b}\in\Syl(\Sp(V_{2b}))$ and choose $Q_{2b}\in\Omega_{2b}^-$ such that $|Q_{2b}^{P_{2b}}|\geq p$. Then $Q:=Q_1\oplus Q_2\oplus \cdots\oplus Q_{2b}\in\Omega^-$. Moreover, letting $\widehat{P}=\Pi_{i=1}^{2b} P_i$ we see that $|Q^{\widehat{P}}|\geq p^{a(2b-1)}p>n_p=p^ab_p$. Thus $G$ has a Sylow $p$-subgroup with an orbit of length greater than $n_p$,  contradicting Lemma \ref{lem:easy}. Next suppose that $e<d$ and $e$ is odd. Note that $b\geq 2$. In this case, let $V=V_1\perp V_2\perp\ldots\perp V_{b}$ where each $V_i$ is a nondegenerate subspace of dimension $2e$ and let $B_i$ be the restriction of $B$ to $V_i\times V_i$. For $i=1,\ldots, b-1$, $\Sp(V_i)$ acts transitively on the set $\Omega_i^+$ of $2^{e-1}(2^e-1)$ elliptic quadratic forms on $V_i$ that polarise to $B_i$.  Let $P_i\in\Syl_p(\Sp(V_i))$. Then by Lemma \ref{lem:shortest} there exists a hyperbolic quadratic form $Q_i\in\Omega_i^-$ such that $|Q_i^{P_i}|=p^a$. Finally, choose $\epsilon_{b}$ so that $(-1)^{b-1}\epsilon_b=-$. Let $P_{b}\in\Syl(\Sp(V_{b}))$ and choose $Q_{b}\in\Omega_{b}^{\epsilon_b}$ such that $|Q_{b}^{P_{b}}|\geq p$. Then $Q:=Q_1\oplus Q_2\oplus \cdots\oplus Q_{b}\in\Omega^-$. Moreover, letting $\widehat{P}=\Pi_{i=1}^{b} P_i$ we see that $|Q^{\widehat{P}}|\geq p^{a(b-1)}p>n_p=p^ab_p$, since $p\geq 3$. Thus $G$ has a Sylow $p$-subgroup with an orbit of length greater than $n_p$,  contradicting Lemma \ref{lem:easy}.  Thus $G$ does not have Property~$(*)_p$ on $\Omega^-$.

\medskip

\item \textsc{Others}

The isolated examples of $2$-transitive groups were dealt with by computer.
First, $\PSL(2,11)$ on 11 points, $M_{11}$ on 11 points, and $M_{23}$ on 23 points
 have prime degree and the result follows from Lemma \ref{lem:elementary}(3).
For the remainder, see Table \ref{tbl:sporadics} for details.

\end{enumerate}

\definecolor{light-gray}{gray}{0.90}

\begin{table}
\begin{tabular}{clcl}
\toprule
$n$ & $G$  & $p$  & Orbit lengths of a Sylow $p$-subgroup\\
\midrule
12 & $M_{11}$ & 2   & $4^{1},8^{1}$\\
\rowcolor{light-gray} & & 3   & $3^{4}$\\
 12 & $M_{12}$ & 2   & $4^{1},8^{1}$\\
 & & 3   & $3^{1},9^{1}$\\
 15 & $A_7$ & 3   & $3^{2},9^{1}$\\
 22 & $M_{22}$ & 2   & $2^{1},4^{1},16^{1}$\\
 22 & $M_{22}:2$ & 2   & $2^{1},4^{1},16^{1}$\\
24 & $M_{24}$ & 2   & $8^{1},16^{1}$\\
 & & 3   & $3^{2},9^{2}$\\
\rowcolor{light-gray} 28 & $\PGammaL(2,8)$ & 2   & $4^{7}$\\
 176 & $\mathrm{HS}$ & 2   & $16^{1},32^{1},64^{2}$\\
276 & $\mathrm{Co}_3$ & 2   & $4^{1},8^{2},16^{2},32^{1},64^{1},128^{1}$\\
 & & 3   & $3^{2},27^{1},243^{1}$\\
\bottomrule\\
\end{tabular}
\caption{Sporadic almost simple $2$-transitive actions.  We only consider the primes $p$ such that  $pn_p < n$ in each case.}\label{tbl:sporadics}
\end{table}


\section{Low degree transitive groups}\label{sec:small}

This section contains computational results on which transitive groups of low degree have Property $(*)_2$ or Property $(*)_3$. This gives an overview of how the theoretical results above are realized in practice. Table \ref{tab:star2} gives summary data for the transitive groups of degrees at most $47$. In this table, $t(n)$ denotes the number of transitive groups of degree $n$. Note that by Lemma \ref{lem:elementary}, a transitive group $G$ of odd degree has Property $(*)_2$ precisely when $G$ has odd order. Similarly, a transitive group of degree $n$ with $n$ coprime to 3 has Property $(*)_3$ precisely when the order of $G$ is coprime to $3$. Moreover, by Lemma \ref{lem:elementary}(4), if $pn_p>n$ then all transitive permutation groups of degree $n$ have Property~$(*)_p$.

\begin{table}[H]
    \centering
    \begin{tabular}{rrrr@{\hskip 1cm}rrrr}
    
    \toprule
    $n$&$t(n)$&$(*)_2$ & $(*)_3$ & $n$&$t(n)$&$(*)_2$ & $(*)_3$\\
    \midrule
    2&1&1&1&3&2&1&2\\
    4&5&5&3& 5&5&1&3\\
    6&16&6&16& 7&7&2&2\\
    8&50&50&27& 9&34&5&34\\
    10&45&5&24& 11&8&2&4\\
    12&301&96&243& 13&9&2&3\\
    14&63&16&14& 15&104&5&66\\
    16&1954&1954&1438& 17&10&1&5\\
    18&983&115&983& 19&8&3&2\\
    20&1117&116&657& 21&164&17&43\\
    22&59&12&32& 23&7&2&4\\
    24&25000&7911&22245& 25&211&10&119\\
    26&96&12&24& 27&2392&231&2392\\
    28&1854&542&461& 29&8&2&6\\
    30&5712&131&4116& 31&12&4&4\\
    32&2801324&2801324&2737818& 33&162&16&100\\
    34&115&7&77& 35&407&12&73\\
    36&121279&9612&1113506& 37&11&3&3\\
    38&76&21&12& 39&306&37&92\\
    40&315842&132071&283122& 41&10&2&8\\
    42&9491&703&2335& 43&10&4&4\\
    44&2113&872&1540& 45&10923&256&7893\\
    46&56&16&36&47&6&2&4\\
    \bottomrule\\
    \end{tabular}
    \caption{Counts of low-degree transitive groups with Property $(*)_2$}
    \label{tab:star2}
\end{table}

\subsection{Maximal \texorpdfstring{$(*)_2$}{(*)2} groups}

As Property $(*)_p$ is closed under taking transitive subgroups of transitive groups, another characterization of the transitive groups with Property $(*)_p$ is to list the \emph{maximal} transitive groups with Property $(*)_p$. Then a transitive group has Property $(*)_p$ if and only if it is a subgroup of one of the listed maximal transitive groups with Property~$(*)_p$.

To illustrate this,
Table~\ref{tab:max} gives the maximal transitive groups with Property $(*)_2$ of degrees $2 \leqslant n \leqslant 39$. Most of these groups are wreath products of primitive groups and their structure can be given explicitly. The few exceptions are given by their indices in the lists of transitive groups in \MAGMA \cite{magma} and \textsf{GAP} \cite{GAP}. So $20_{89}$ is the group returned by the expression {\tt TransitiveGroup(20,89)} in either \MAGMA or \textsf{GAP}.

We note that for $n=2^k$, every transitive group has Property $(*)_2$ and hence the unique maximal transitive $(*)_2$ is the symmetric group $S_{2^k}$. Every transitive group of odd order has odd degree and has Property $(*)_2$. Hence for each odd degree, the table simply lists the maximal transitive groups of odd order.

\begin{longtable}{ll}
\toprule
Degree & \multicolumn{1}{c}{Structures of the maximal $(*)_2$ groups} \\
\midrule
\endhead
2 & $2$ \\
3 & $3$\\
4 & $S_4$\\
5 & $5$\\
6 & $3 \Wr 2$,\: $2 \Wr 3$,\: $\PSL(2,5)$\\
7 & $7:3$ \\
8 & $S_8$ \\
9 & $3 \Wr 3$ \\
\midrule
10 & $5 \Wr 2$,\:  $2 \Wr 5$ \\
11 & $11:5$ \\
12 & $\mathrm{PSL}(2,11)$,\:  
    $2 \Wr (3 \Wr 2)$,\:  
    $2 \Wr \PSL(2,5)$,:  
    $3 \Wr S_4$,\: 
    $\mathrm{PSL}(2,5) \Wr 2$,\: 
    $S_4 \Wr 3$ \\
13 & $13:3$ \\
14 & $(7 : 3) \Wr 2$,\:  $2 \Wr (7:3)$ \\
15 & $5 \Wr 3$,\:  $3 \Wr 5$ \\
16 & $S_{16}$ \\
17 & $17$ \\
18 & 
$3 \Wr 3 \Wr 2$,\:  $3 \Wr 2 \Wr 3$,\:  $2 \Wr 3 \Wr 3$,\:  $3 \Wr \PSL(2,5)$,\:  $\PSL(2,5) \Wr 3$ \\
19 & $19:9$ \\
\midrule
20 & 
    $20_{89} \leqslant 2 \Wr \PSL(2,9)$,\: 
    $\PSL(2,19)$,\:  $5 \Wr S_4$,\:  $2 \Wr 5 \Wr 2$,\:  $S_4 \Wr 5$ \\
21 &
    $(7:3) \Wr 3$,\:  $3 \Wr (7:3)$ \\
22 & $(11:5) \Wr 2$,\:  $2 \Wr (11:5)$ \\
23 & $23 : 11$ \\
24 & $\PSL(2,23)$,\:  $\PSL(2,11) \Wr 2$,\:  $2 \Wr \PSL(2,11)$,\:  $2 \Wr 3 \Wr S_4$,\:  $2 \Wr \PSL(2,5) \Wr 2$,\: \\ & $3 \Wr S_8$,\:  $\PSL(2,5) \Wr S_4$,\:  $S_4 \Wr 3 \Wr 2$,\:  $S_4 \Wr \PSL(2,5)$,\:  $S_8 \Wr 3$ \\
25 & $5^2:3$,\:  $5 \Wr 5$ \\
26 & $(13:3) \Wr 2$,\:  $2 \Wr (13:3)$ \\
27 & $3^3:13.3$,\:  $3 \Wr 3 \Wr 3$ \\
28 & $28_{120} \leqslant 2 \Wr \PSL(2,13)$,\: $\PGammaL(2,8)$,\: $\PSL(2, 27):3$,\: $(7:3) \Wr S_4$\\& $2 \Wr (7:3) \Wr 2$,\: $S_4 \Wr (7:3)$ \\
29 & $29:7$ \\
\midrule
30 & $5 \Wr 3 \Wr 2$,\: $5 \Wr 2 \Wr 3$,\: $5 \Wr \text{PSL}(2,5)$,\: $3 \Wr 5 \Wr 2$,\:$3 \Wr 2 \Wr 5$,\: $2 \Wr 5 \Wr 3$,\:\\&  $2 \Wr 3 \Wr 5$,\: $\text{PSL}(2,5) \Wr 5$ \\
31 & $31:15$ \\
32 & $S_{32}$ \\
33 & $(11:5) \Wr 3$,\: $3 \Wr (11:5)$ \\
34 &  $17 \Wr 2$,\: $2 \Wr 17$ \\
35 & $5 \Wr (7:3)$,\: $(7:3) \Wr 5$ \\
36 & 
$3 \Wr \PSL(2,11)$,\:
$3 \Wr 2 \Wr 3 \Wr 2$,\:
$\PSL(2,11) \Wr 3$,\:
$3 \Wr 3 \Wr S_4$,\:
$3 \Wr 2 \Wr \PSL(2,5)$,
\\ &
$2 \Wr 3 \Wr 3 \Wr 2$,\:
$3 \Wr \PSL(2,5) \Wr 2$,\:
$2 \Wr 3 \Wr 2 \Wr 3$,\:
$2 \Wr 3 \Wr \PSL(2,5)$,
\\ &
$3 \Wr S_4 \Wr 3$,\:
$2 \Wr \PSL(2,5) \Wr 3$,\:
$\PSL(2,5) \Wr 3 \Wr 2$,
\\&
$\PSL(2,5) \Wr 2 \Wr 3$,\: 
$\PSL(2,5) \Wr \PSL(2,5)$, 
$S_4 \Wr 3 \Wr 3$
\\
37 & $37:9$ \\
38 & $(19 : 9) \Wr 2$ \\
39 & $(13 : 3) \Wr 3$,\:  $3 \Wr (13 : 3)$ \\
\bottomrule
\caption{Structure of maximal $(*)_2$ groups}
\label{tab:max}
\end{longtable}

Many of these maximal groups are wreath products involving maximal $(*)_2$-groups of smaller degrees. By Lemma \ref{lem:imprimwreath}, if $a_n$ and $b_m$ have Property $(*)_2$ then $a_n\Wr b_m$ has Property~$(*)_2$ in its imprimitive action of degree $ab$. However, even if $a_n$ and $b_m$ are maximal $(*)_2$-groups of degree $a$ and $b$ respectively, it does not necessarily follow that $a_n \Wr b_m$ is a maximal $(*)_2$ group in $S_{ab}$. The smallest example where this occurs is $(3 \Wr 2) \Wr 2$ which has Property $(*)_2$, but is not maximal. By associativity of the wreath product $(3 \Wr 2) \Wr 2 = 3 \Wr (2 \Wr 2)$, but $2 \Wr 2$ is not maximal, as it is contained in $S_4$. 

In some sense the most interesting maximal transitive $(*)_2$ groups are the imprimitive ones that are not wreath products as these are examples where the permutation group induced on blocks does not have Property $(*)_2$. This occurs just twice in the table --- the group $20_{89}$ is a subgroup of the group $2 \Wr \PSL(2,9)$, and the group $28_{120}$ is a subgroup of the group $2 \Wr \PSL(2,13)$, although neither $\PSL(2,9)$ or $\PSL(2,13)$ have Property $(*)_2$.

\subsection*{Acknowledgements}
The work on this paper began at the 2020 annual research retreat of the Centre for the Mathematics of Symmetry and Computation. The authors thank the participants for the enjoyable and stimulating environment.


\end{document}